\documentclass[11pt]{amsart}
\usepackage{amssymb}
\usepackage{amscd}
\usepackage{amsmath}
\usepackage[all]{xy}
\usepackage{mathrsfs}
 \usepackage{color}

\usepackage[dvipdfmx]{hyperref}

\usepackage{bm}

\calclayout

\theoremstyle{plain}
\newtheorem{theorem}{Theorem}[section]
\newtheorem{theorem*}{Theorem}
\newtheorem{proposition}[theorem]{Proposition}
\newtheorem{lemma}[theorem]{Lemma}
\newtheorem{corollary}[theorem]{Corollary}
\newtheorem{conjecture}[theorem]{Conjecture}

\theoremstyle{remark}
\newtheorem{remark}[theorem]{Remark}

\theoremstyle{definition}
\newtheorem{definition}[theorem]{Definition}

\newtheorem{hypothesis}[theorem]{Hypothesis}

\newtheorem*{acknowledgments}{Acknowledgements}

\font\russ=wncyr10  1
\def\sha{\hbox{\russ\char88}}

\DeclareMathOperator{\Ext}{Ext}
\DeclareMathOperator{\Gal}{Gal}

\DeclareMathOperator{\Hom}{Hom}

\DeclareMathOperator{\Spec}{Spec}
\DeclareMathOperator{\N}{N}

\DeclareMathOperator{\coker}{coker}

\newcommand{\bA}{\mathbb{A}}

\newcommand{\TT}{\mathbb{T}}

\newcommand{\QQ}{\mathbb{Q}}

\newcommand{\cE}{\mathcal{E}}
\newcommand{\cF}{\mathcal{F}}

\newcommand{\cL}{\mathcal{L}}

\newcommand{\cO}{\mathcal{O}}

\newcommand{\fq}{\mathfrak{q}}
\newcommand{\fp}{\mathfrak{p}}

\newcommand{\fz}{\mathfrak{z}}

\newcommand{\CC}{\mathbb{C}}

\newcommand{\FF}{\mathbb{F}}
\newcommand{\GG}{\mathbb{G}}

\newcommand{\RR}{\mathbb{R}}

\newcommand{\ZZ}{\mathbb{Z}}

\newcommand{\id}{\mathrm{id}}

\newcommand{\ord}{\mathrm{ord}}

\newcommand{\rgamma}{\mathbf{R}\Gamma}
\newcommand{\rhom}{\mathbf{R}\Hom}
\newcommand{\lotimes}{\otimes^{\mathbf{L}}}

\newcommand{\barfp}{\overline{\mathfrak{p}}}

\makeatletter
 
  \@addtoreset{equation}{subsection}
\makeatother

\begin{document}

\title[]{On the Tamagawa number conjecture for modular forms twisted by anticyclotomic Hecke characters}

\author{Takamichi Sano}

\begin{abstract}
Let $f \in S_{2r}(\Gamma_0(N))$ be a normalized newform of weight $2r$ which is good at $p$. Let $K$ be an imaginary quadratic field of class number one in which every prime divisor of $pN$ splits. Let $\chi$ be an anticyclotomic Hecke character of $K$ which is crystalline at the primes above $p$ and such that $L(f,\chi,r)\neq 0$. 
We prove that the Tamagawa number conjecture for the critical value $L(f,\chi,r)$ follows from the Iwasawa main conjecture for the Bertolini-Darmon-Prasanna $p$-adic $L$-function. 
\end{abstract}

\address{Osaka Metropolitan University,
Department of Mathematics,
3-3-138 Sugimoto\\Sumiyoshi-ku\\Osaka\\558-8585,
Japan}
\email{tsano@omu.ac.jp}

\maketitle

\tableofcontents

\section{Introduction}

\subsection{Background}

Understanding the special values of $L$-functions is a fundamental problem in number theory. The Tamagawa number conjecture of Bloch-Kato \cite{BK}, generalized by Kato \cite{katokodai}, \cite{katolecture}, Fontaine and Perrin-Riou \cite{fontaineL}, \cite{FP} and Burns-Flach \cite{BFetnc}, is the most general and sophisticated conjecture on the special values of $L$-functions. This conjecture includes the Birch and Swinnerton-Dyer conjecture as a special case, which still remains largely open, and is considered extremely difficult.

The first significant contribution to the (equivariant) Tamagawa number conjecture was given by Burns-Greither \cite{BG} and Huber-Kings \cite{HK}. They proved the conjecture for Tate motives associated with abelian fields. An important aspect of their result is that they showed that Iwasawa theory can be effectively used to tackle the conjecture. Their method belongs to what is known as ``descent theory". It would be worth noting that the basic idea behind such a method can already be seen in \cite[\S 6]{BK}.

More recently, their method was generalized by Burns, Kurihara and the present author in \cite{bks2} to general base number fields. They developed descent theory and gave a strategy for proving the (equivariant) Tamagawa number conjecture for Tate motives over general number fields. An analogue of this work for elliptic curves was later given in \cite{bks4}. Moreover, Kataoka and the present author \cite{ks} studied a generalization of \cite{bks4}, and gave a general strategy for proving the Tamagawa number conjecture for a general motive. The key to this work is to consider (higher rank) Euler systems for a general motive, whose existence was predicted by a conjecture formulated by Burns, Sakamoto and the present author \cite{bss2}. 

The underlying philosophy behind these works is that {\it one should consider Iwasawa theory of Euler systems in order to approach the Tamagawa number conjecture}. On the other hand, in Iwasawa theory, $p$-adic $L$-functions are intensively studied, and a variety of new constructions have been found. However, it seems that the relation between $p$-adic $L$-functions and the Tamagawa number conjecture has not been thoroughly investigated. In particular, an application of the Bertolini-Darmon-Prasanna (BDP) $p$-adic $L$-function (constructed in \cite{BDP}, \cite{brakocevic}, \cite{CH}) to the Tamagawa number conjecture has hardly been studied so far. (However, we remark that it has been applied to the Birch and Swinnerton-Dyer conjecture, most notably in the work of Jetchev-Skinner-Wan \cite{JSW}.)

In this article, we make a first attempt to apply the BDP $p$-adic $L$-function to the Tamagawa number conjecture for modular forms twisted by anticyclotomic Hecke characters. 

\subsection{Main results}

We set some notation. 
Let $f = \sum_{n=1}^\infty a_n q^n \in S_{2r}(\Gamma_0(N))$ be a normalized newform of weight $2r\geq 2$ and level $N$. Let $p$ be an odd prime number such that $p\nmid N$. Let $K$ be an imaginary quadratic field of class number one with odd discriminant $-D_K <-3$. We assume that every prime divisor of $N$ splits in $K$ (Heegner hypothesis). We also assume that $p$ splits in $K$: we write $(p)=\fp\barfp$. 
Let $\overline \QQ$ be the algebraic closure of $\QQ$ in $\CC$ and fix an embedding $\iota_p: \overline \QQ \hookrightarrow \CC_p$ which induces $K_\fp \xrightarrow{\sim} \QQ_p$. Let $F$ be a finite extension of $\QQ_p$ which contains $\iota_p(a_n)$ for all $n\geq 1$. Let $V_{f,F}$ be the $p$-adic Galois representation of $G_\QQ:=\Gal(\overline \QQ/\QQ)$ attached to $f$ with coefficients in $F$. 
Let $K_\infty/K$ be the anticyclotomic $\ZZ_p$-extension and set $\Gamma:=\Gal(K_\infty/K)$. 
Let $\chi : \Gamma \to \cO_F^\times$ be a character which is crystalline at both $\fp$ and $\barfp$ and corresponds to an anticyclotomic Hecke character $\chi_\bA: \bA_K^\times/K^\times \to \CC^\times$ of infinity type $(j,-j)$. 
By a slight abuse of notation, we denote $\chi_\bA$ also by $\chi$.

Let $M_f$ be the motive attached to $f$. We consider the critical motive $M:=M_f(r)\otimes \chi$ defined over $K$, whose $p$-adic realization is the $G_K$-representation $V_{f,F}(r) \otimes \chi$. The $L$-function attached to $M$ is $L(M,s)=L(f,\chi^{-1},s+r)$. We consider the Tamagawa number conjecture of Bloch-Kato \cite{BK} for the pair $(M,\cO_F)$ (in the sense of Burns-Flach \cite[Conj. 4(iv)]{BFetnc}), which determines the leading term of $L(M,s)$ at $s=0$ (i.e., the leading term of $L(f,\chi^{-1},s)$ at $s=r$) up to a unit in $\cO_F^\times$. (See Conjecture \ref{TNC} for the precise statement.)

The main result of this article is the following. 

\begin{theorem}[= Theorem \ref{thm:main}]\label{main}
Assume $L(f,\chi^{-1},r)\neq 0$. 
Then the Tamagawa number conjecture for the pair $(M_f(r)\otimes \chi,\cO_F)$ is implied by the Iwasawa main conjecture for the Bertolini-Darmon-Prasanna $p$-adic $L$-function for $f$.
\end{theorem}

\begin{remark}\label{jrem}
The condition $L(f,\chi^{-1},r)\neq 0$ implies that the sign of the functional equation is $+1$, which is equivalent to $j\geq r$ or $j\leq -r$ (see \cite[Rem. on p.569]{CH}). 
\end{remark}

\begin{remark}\label{rem CHK}
The ``order of vanishing" part of the Tamagawa number conjecture
$$ L(f,\chi^{-1},r) \neq 0 \Rightarrow H^1_f(K, V_{f,F}(r)\otimes \chi^{-1})=0$$
(cf. \cite[Conj. 4(ii)]{BFetnc}) is proved by Castella-Hsieh \cite[Thm. A]{CH} (resp. Kobayashi \cite{kobayashi1}, \cite{kobayashi2}) when $f$ is ordinary (resp. supersingular) at $p$. 
\end{remark}

In the case of elliptic curves (i.e., when $r=1$ and $f$ corresponds to an elliptic curve $E$ over $\QQ$), the Iwasawa main conjecture is proved in \cite[Thm. B]{BCK} and \cite[Cor. 7.2]{CHKLL} under mild hypotheses. Thus we obtain the following unconditional result on the Tamagawa number conjecture. 

\begin{corollary}[= Corollary \ref{cor:main}]
Suppose $r=1$ and $f$ corresponds to an elliptic curve $E$. Assume the following:
\begin{itemize}
\item $L(f,\chi^{-1},1)\neq 0$;
\item the representation $\rho:G_\QQ\to {\rm Aut}(E[p])$ is surjective;
\item $\rho$ is ramified at every $\ell | N$;
\item 
$N$ is square-free;
\item $p$ is non-anomalous.
\end{itemize}
Then the Tamagawa number conjecture for $(h^1(E/K)(1)\otimes \chi, \cO_F)$ is true. 
\end{corollary}

\begin{remark}\label{rem number one}
Although we assume that $K$ has class number one in this article, we do not think this is essential. The main reason why we assume this is that the motive attached to $\chi$ is simply described in this case. By using the idea of this article, it would be possible to prove Theorem \ref{main} without the class number one assumption. We remark that the class number one assumption is often made in the works on the Tamagawa number conjecture for Hecke characters (such as \cite{kings}, \cite{tsuji}, \cite{bars}). 
\end{remark}

\subsection{Idea of the proof}\label{sec idea}
We shall sketch an idea of the proof of Theorem \ref{main}. 

Fix a finite set $S$ of places of $K$ containing the infinite place and the primes dividing $pN$. Let $G_{K,S}$ be the Galois group of the maximal extension of $K$ unramified outside $S$. Let $T_f$ be a stable lattice of the $p$-adic Galois representation attached to $f$. Let $\Lambda$ be the anticyclotomic Iwasawa algebra and set $\Lambda^{\rm ur}:=\widehat \ZZ_p^{\rm ur}\widehat \otimes \Lambda$, where $\widehat \ZZ_p^{\rm ur}$ denotes the completion of the ring of integers of the maximal unramified extension of $\QQ_p$. We consider the deformation $\TT:=\Lambda \otimes T_f(r)$.

The key idea is to construct an element
\begin{equation}\label{zs}
\fz_S \in {\det}_\Lambda^{-1}(\rgamma(G_{K,S},\TT))
\end{equation}
which ``interpolates" special values $L(f,\chi^{-1},r)$. 
In the construction, we use the Bertolini-Darmon-Prasanna $p$-adic $L$-function 
$$L_\fp^{\rm BDP} \in \Lambda^{\rm ur}$$
and the ``local epsilon element"
$$\varepsilon_\fp \in \Lambda^{\rm ur}\otimes_\Lambda {\det}_\Lambda(\rgamma(K_\fp,\TT)),$$
whose existence is predicted by Kato's local epsilon conjecture \cite{katolecture2}. (In our case, this conjecture is proved by Loeffler-Venjakob-Zerbes \cite{LVZ}, Nakamura \cite{nakamura} and Rodrigues~Jacinto \cite{RJ}.) By the general philosophy of Fukaya-Kato in \cite[\S 4]{FK}, local epsilon elements should play the role of ``Coleman maps". Roughly speaking, we define $\fz_S$ to be the preimage of $L_\fp^{\rm BDP}$ under the ``Coleman map" $\varepsilon_\fp$. A delicate point is to check that $\fz_S$ has coefficients in $\Lambda$, and to do this we need to study the behavior of $L_\fp^{\rm BDP}$ under the Frobenius action (see Lemma \ref{phiBDP}). The Iwasawa main conjecture ensures that $\fz_S$ is a $\Lambda$-basis. See \S \ref{sec construction} for the precise construction of $\fz_S$. It would be worth noting that $\fz_S$ is essentially an Euler system of rank two: see \S \ref{subsec euler} below. 

The proof of Theorem \ref{main} is then reduced to checking that $\fz_S$ has the desired interpolation property predicted by the Tamagawa number conjecture. (This is the content of  Theorem \ref{thm:interpolation}.) To do this we use the interpolation property of $L_\fp^{\rm BDP}$:
$$\chi(L_\fp^{\rm BDP}) =\Omega_p^{4j}\cdot \Gamma(j-r+1) \Gamma(j+r)\cdot  (1-a_p \chi^{-1}(\barfp) p^{-r} + \chi^{-2}(\barfp) p^{-1})^2 
 \frac{L(f,\chi^{-1},r)}{ \Omega_\infty^{4j} (2\pi)^{1-2j}\sqrt{D_K}^{2j-1} }.$$
Here $\Omega_\infty \in \CC^\times$ and $\Omega_p \in (\widehat \ZZ_p^{\rm ur})^\times$ denote complex and $p$-adic CM periods respectively. We check that the $p$-adic period $\Omega_p$ and the $\Gamma$-factor $\Gamma(j-r+1) \Gamma(j+r)$ arise in the interpolation property of $\varepsilon_\fp$. We also check that $\Omega_\infty^{4j} (2\pi)^{1-2j}\sqrt{D_K}^{2j-1} $ is the period of the motive $M_f(r)\otimes \chi$ in the sense of Deligne \cite{deligne}. Such comparison of periods and calculations will be made in \S \ref{sec compare}. 
(Here we use the simplifying assumption that $K$ has class number one.)

We remark that, in a forthcoming article \cite{sanopadic}, we generalize the method used in the present work. More precisely, we consider a $p$-adic $L$-function for a general motive and give a general strategy for proving the Tamagawa number conjecture using the $p$-adic $L$-function.

\subsection{Related topics}

We shall discuss some related topics.

\subsubsection{The Tamagawa number conjecture in analytic rank one}
It is natural to also consider Hecke characters $\chi$ such that $\ord_{s=r}L(f,\chi^{-1},s)=1$. (This happens when $-r < j <r$.) In fact, Bertolini-Darmon-Prasanna proved a ``$p$-adic Gross-Zagier formula" which relates $\chi(L_\fp^{\rm BDP})$ for such $\chi$ with generalized Heegner cycles (see \cite[Thm. 5.13]{BDP}). Combining this result with a generalization of the Gross-Zagier formula (as in \cite{zhang}, \cite{CST}, \cite{YZZ}, \cite{LS}), it would be possible to prove that the element (\ref{zs}) interpolates $L'(f,\chi^{-1},r)$, which would lead to a proof of the Tamagawa number conjecture in analytic rank one (under the Iwasawa main conjecture). We do this in Appendix \ref{sec app} in the special case when $\chi$ is the trivial character and $f$ corresponds to an elliptic curve. In particular, we give another proof of a result essentially obtained by Jetchev-Skinner-Wan \cite{JSW} on the Birch and Swinnerton-Dyer conjecture in analytic rank one (see Theorem \ref{thman1} and Remark \ref{rem JSW}). We also give a similar result when $p$ is inert in $K$ (see Theorem \ref{thmheeg}). 

It would also be interesting to compare our method with the recent work by Castella \cite{castellaTNC}, where the Bertolini-Darmon-Prasanna $p$-adic $L$-function is applied to the Tamagawa number conjecture for CM elliptic curves in analytic rank one. 

\subsubsection{Euler systems of rank two}\label{subsec euler}
The element $\fz_S$ we construct in (\ref{zs}) is essentially an Euler system of rank two. In fact, by \cite[Thm. 2.18]{sbA} (see also \cite[Lem. 3.11(ii)]{ks}), there is a canonical map
\begin{equation}\label{Theta}
\Theta: {\det}_\Lambda^{-1}(\rgamma(G_{K,S},\TT)) \to {\bigcap}_\Lambda^2 H^1(G_{K,S},\TT),
\end{equation}
where for $r\geq 0$ we write $\bigcap_\Lambda^r$ for the $r$-th ``exterior power bidual" introduced in \cite{sbA}. Using the element (\ref{zs}), we define 
\begin{equation}\label{zinfty}
z_{K_\infty} := \Theta(\fz_S) \in {\bigcap}_\Lambda^2 H^1(G_{K,S},\TT).
\end{equation}
We conjecture that this is (at least up to a certain normalization) the ``$K_\infty$-component" of the Euler system whose existence is predicted by a general conjecture \cite[Conj. 2.6]{ks} (see also \cite[\S 4]{bss2}). Main results of the present work can be regarded as a partial solution to this conjecture. 

We remark that we can unconditionally construct $z_{K_\infty}$ as an element of $Q(\Lambda)\otimes_\Lambda {\bigwedge}_\Lambda^2 H^1(G_{K,S},\TT)$ ($Q(\Lambda) $ denotes the quotient field of $\Lambda$), and the Iwasawa main conjecture for $L_\fp^{\rm BDP}$ is equivalent to the equality
\begin{equation}\label{IMC without L}
{\rm char}_\Lambda\left({\bigcap}_\Lambda^2 H^1(G_{K,S},\TT) / \Lambda\cdot z_{K_\infty} \right) = {\rm char}_\Lambda(H^2(G_{K,S},\TT)).
\end{equation}
(See \cite[Prop. 3.10]{ks}.)

We also remark that our construction of $z_{K_\infty}$ is different from the construction of the ``$\Lambda$-adic Heegner element" $z_{\infty}^{\rm Hg}$ in \cite[\S 5.2.3]{ks}, since $z_{K_\infty}$ is constructed by using the Bertolini-Darmon-Prasanna $p$-adic $L$-function, while $z_\infty^{\rm Hg}$ by Heegner points (in the elliptic curve case). (In the former case we assume $p$ splits, while in the latter case $p$ does not necessarily split but it must be ordinary.) Note also that the construction of $z_\infty^{\rm Hg}$ in \cite{ks} is non-canonical. In \S \ref{sec ks}, we improve the construction of $z_\infty^{\rm Hg}$ by using a local epsilon element (see Remark \ref{rem heeg element}). 

\subsubsection{$p$-adic Birch and Swinnerton-Dyer conjectures}

In \cite{AC}, Agboola-Castella formulated a $p$-adic analogue of the Birch and Swinnerton-Dyer conjecture for the Bertolini-Darmon-Prasanna $p$-adic $L$-function. On the other hand, the present author formulated a conjecture on derivatives of (higher rank) Euler systems for motives in \cite[Conj. 5.5]{sanoderived}. In particular, if we specialize it to the rank two Euler system $z_{K_\infty}$ in (\ref{zinfty}), we can formulate an analogue of the $p$-adic Birch and Swinnerton-Dyer conjecture concerning derivatives of $z_{K_\infty}$ (as in \cite[Prop. 5.15]{sanoderived}). In a forthcoming work \cite{sanopadic}, we prove that this conjecture is equivalent to the conjecture of Agboola-Castella. We also prove that, in the ordinary case, it is equivalent to the $p$-adic Birch and Swinnerton-Dyer conjecture for Heegner points formulated by Bertolini-Darmon \cite{BD}. By these results, we can view both the Agboola-Castella and Bertolini-Darmon conjectures as special cases of the general conjecture for motives.

\subsection{General notation}

For a commutative ring $R$ and an $R$-module $X$, we set
$$X^\ast:=\Hom_R(X,R).$$
If $X$ is a free $R$-module of rank one and $x \in X$ is a basis, then the dual basis 
$$x^\ast \in X^\ast$$
is the $R$-linear map satisfying $x^\ast( x) =1$. 
If $X$ is a $\ZZ_p$-module, then its Pontryagin dual is denoted by
$$X^\vee:=\Hom_{\ZZ_p}(X,\QQ_p/\ZZ_p).$$

Let ${\rm det}_R$ denote the determinant functor of Knudsen-Mumford \cite{KM}. This functor associates to a perfect complex $C$ of $R$-modules a graded invertible $R$-module ${\rm det}_R(C)$. We often regard ${\rm det}_R(C)$ as an invertible $R$-module by forgetting the grading part. In particular, taking an element of ${\rm det}_R(C)$ makes sense. Note that, due to the sign issue, we need to regard ${\det}_R(-)$ as graded invertible modules when we use an isomorphism between ${\det}_R(C)\otimes_R {\det}_R(D)$ and $ {\det}_R(D)\otimes_R {\det}_R(C)$ for two perfect complexes $C$ and $D$. 

In this article, we use ${\det}_R$ when $R$ is a regular local ring (e.g., any field, any discrete valuation ring, $\ZZ_p[[T]]$, etc.). In this case, every bounded complex of finitely generated $R$-modules is perfect. In particular, we can define ${\rm det}_R(X)$ for any finitely generated $R$-module $X$ by identifying $X$ with its projective resolution. For basic properties of ${\det}_R$, see \cite[\S 1.3]{sanoderived} for example. 

For any field $L$, its absolute Galois group is denoted by $G_L$.

Let $\overline \QQ$ be the algebraic closure of $\QQ$ in $\CC$. Any number field $K$ (i.e., a finite extension of $\QQ$) is regarded as a subfield of $ \CC$. The ring of integers of $K$ is denoted by $\cO_K$. 


For a number field $K$ and a finite place $v$ of $K$, we fix a place of $\overline \QQ$ lying above $v$. We regard $G_{K_v} $ as the decomposition subgroup of $v$ in $G_K$. The maximal unramified extension of $K_v$ is denoted by $K_v^{\rm ur}$. The inertia subgroup of $v$ in $G_K$ is defined by $I_v:=G_{K_v^{\rm ur}}$. The arithmetic Frobenius of $v$ is denoted by ${\rm Fr}_v \in \Gal(K_v^{\rm ur}/K_v)=G_{K_v}/I_v$.  


We use the standard notation of (continuous) Galois cohomology. For the definitions of $H^1_f$ and $H^1_{\rm ur}$, see \cite[\S 1.3]{R} for example. We set $H^1_{/f}:=H^1/H^1_f$ and $H^1_{/{\rm ur}}:=H^1/H^1_{\rm ur}$. $H^0_f$ is understood to be $H^0$. For a finite place $v$ and a $G_{K_v}$-module $X$, the unramified cohomology complex is defined by 
$$\rgamma_{\rm ur}(K_v, X):=\rgamma(K_v^{\rm ur}/K_v, X^{I_v}) = [ X^{I_v}\xrightarrow{1-{\rm Fr}_v^{-1}} X^{I_v}].$$
Let ${\rm Inf}_{K_v^{\rm ur}/K_v}: \rgamma_{\rm ur}(K_v,X) \to \rgamma(K_v,X)$ be the inflation morphism and set 
$$\rgamma_{/{\rm ur}}(K_v,X):= {\rm Cone}\left(\rgamma_{\rm ur}(K_v,X)\xrightarrow{-{\rm Inf}_{K_v^{\rm ur}/K_v}} \rgamma(K_v,X)\right).$$

Let $S$ be a finite set of places of $K$ containing all infinite and $p$-adic places of $K$. Let $S_f\subset S$ be the subset of finite places. Let $K_S/K$ be the maximal extension unramified outside $S$ and set $G_{K,S}:=\Gal(K_S/K)$. Let $V$ be a finite dimensional $\QQ_p$-vector space endowed with a continuous linear action of $G_{K,S}$. We frequently use the Poitou-Tate exact sequence
\begin{multline*}
0\to H^1_f(K,V)\to H^1(G_{K,S},V)\to \bigoplus_{v\in S_f}H^1_{/f}(K_v,V)\\
\to H^1_f(K,V^\ast(1))^\ast \to H^2(G_{K,S},V)\to \bigoplus_{v\in S_f}H^0(K_v,V^\ast(1))^\ast.
\end{multline*}
(See \cite[Prop. II.2.2.1]{FP} for example.) 

\section{Statement of the main results}

The aim of this section is to state the main results precisely (Theorem \ref{thm:main} and Corollary \ref{cor:main}). In \S \ref{sec TNC}, we review the formulation of the Tamagawa number conjecture in the setting as in Introduction. In \S \ref{sec IMC}, after reviewing the Bertolini-Darmon-Prasanna $p$-adic $L$-function and the Iwasawa main conjecture, we state the main results. 

Throughout this article, let $p$ be an odd prime number. Let $f = \sum_{n=1}^\infty a_n q^n \in S_{2r}(\Gamma_0(N))$ be a normalized newform of weight $2r\geq 2$ and level $N$. We assume $p\nmid N$. Let $\cF$ be a number field which contains all $a_n$. Let $\lambda$ be the prime of $\cF$ lying above $p$ corresponding to the fixed embedding $\iota_p :\overline \QQ\hookrightarrow \overline \QQ_p$. We set $F:=\cF_\lambda$. Let $V_{f,F}$ be the $\lambda$-adic Galois representation of $G_\QQ$ attached to $f$ with coefficients in $F$. 

Let $K$ be an imaginary quadratic field of class number one with odd discriminant $-D_K <-3$\footnote{This means $D_K \in \{7,11,19,43,67,163\}$.}. We assume that every prime divisor of $pN$ splits in $K$. We write $(p)=\fp\barfp$ in $K$ so that $\fp$ corresponds to $\iota_p$. 


Let $\chi : \bA_K^\times/K^\times \to \CC^\times$ be an anticyclotomic Hecke character of infinity type $(j,-j)$ which takes values in $\cF$. (See \S \ref{sec ex hecke} for the definition of Hecke characters.) By Remark \ref{jrem}, we are interested in the case $j\geq r$ or $j\leq -r$. We may assume $j\geq r$, since the other case is treated in the same way by considering $\overline \chi$ instead of $\chi$. The $p$-adic avatar of $\chi$ is a character $G_K\to \cO_F^\times$, which is also denoted by $\chi$. 

To simplify the notation, we set
$$V:=V_{f,F}(r)\otimes \chi^{-1}.$$
Let $S$ be a finite set of places of $K$ which contains the infinite place, the $p$-adic primes, and the primes at which $V$ ramify. Let $G_{K,S}$ be the Galois group of the maximal Galois extension of $K$ unramified outside $S$. Let $V^\ast(1):=\Hom_F(V,F(1)) \simeq V_{f,F}(r)\otimes \chi$ be the Kummer dual of $V$. 

In this article, we regard a motive as a collection of realizations and comparison isomorphisms (see Appendix \ref{sec motive} for details). 
Let $M_f$ and $M(\chi)$ be the motives attached to $f$ and $\chi$ respectively (see \S\S \ref{sec ex modular} and \ref{sec ex hecke}). We regard $M_f$ as a motive defined over $K$ (see Remark \ref{base extension}). We consider the motive
$$M:=M_f(r)\otimes M(\chi),$$
which is defined over $K$ of rank two with coefficients in $\cF$. The weight of $M$ is $-1$ and so $M$ is critical. Note that its $\lambda$-adic realization is $V^\ast(1)$, and the $L$-function is $L(M,s)=L(f,\chi^{-1},s+r)$.

\subsection{The Tamagawa number conjecture}\label{sec TNC}

In this subsection, we review the statement of the Tamagawa number conjecture \cite[Conj. 4(iv)]{BFetnc} for the pair $(M,\cO_F)$ in analytic rank zero. (This is a special case of Conjecture \ref{TNC general0}.) 


\begin{lemma}\label{lem:fp}
Let $H^1_f(K_v, V) \subset H^1(K_v,V)$ denote the Bloch-Kato local condition for $v \in \{\fp, \barfp\}$. Then we have
$$H^1_f(K_\fp, V)=0 \text{ and } H^1_f(K_{\barfp}, V) =H^1(K_{\barfp}, V).$$
\end{lemma}

\begin{proof}
Since the Hodge-Tate weights\footnote{Our convention is that the Hodge-Tate weight of $\QQ_p(1)$ is $+1$.} of $V_{f,F}$ are $0$ and $1-2r$, we see that the Hodge-Tate weights of $V=V_{f,F}(r)\otimes \chi^{-1}$ at $\fp$ (resp. $\barfp$) are $r-j$ and $1-r-j$ (resp. $r+j$ and $1-r+j$). Since $j\geq r$, the claim follows from \cite[Thm. 4.1(ii)]{BK}. 
\end{proof}

\begin{corollary}\label{cor:fp}
Assume the Bloch-Kato Selmer groups $H^1_f(K,V)$ and $H^1_f(K,V^\ast(1))$ vanish. Then we have $H^2(G_{K,S},V)=0$ and the localization map at $\fp$ induces an isomorphism:
$${\rm loc}_\fp: H^1(G_{K,S}, V)\xrightarrow{\sim} H^1(K_\fp, V). $$
In particular, we have $\dim_F(H^1(G_{K,S},V))=2$. 
\end{corollary}

\begin{proof}
We have the Poitou-Tate exact sequence
$$0\to H^1_f(K,V)\to H^1(G_{K,S},V)\to \bigoplus_{v\in S}H^1_{/f}(K_v,V)\to H^1_f(K,V^\ast(1))^\ast \to H^2(G_{K,S},V)\to 0.$$
By Lemma \ref{lem:fp} and the fact that $H^1_{/f}(K_v,V)=0$ for $v \nmid p$, we have 
$$\bigoplus_{v\in S}H^1_{/f}(K_v,V) = H^1(K_\fp, V).$$
The claim follows immediately from this. 
\end{proof}


Note that the Betti realization $H_B(M)$ and the de Rham realization $H_{\rm dR}(M)$ of $M$ are $\cF$-vector spaces of dimension four. 
We have the comparison isomorphisms:
$$ \CC\otimes_\QQ H_B(M)\xrightarrow{\sim} \CC\otimes_\QQ H_{\rm dR}(M),$$
$$F\otimes_\cF H_B(M)^+ \simeq V^\ast(1),$$
$$F\otimes_\cF H_{\rm dR}(M) \simeq D_{{\rm dR},\fp}(V^\ast(1))\oplus D_{{\rm dR},\barfp}(V^\ast(1)).$$
Here we set $D_{{\rm dR},v}(-):=H^0(K_v, B_{\rm dR}\otimes_{\QQ_p}-)$ for $v\in \{\fp, \barfp\}$. 
The tangent space of $M$ is defined by 
$$t(M):=H_{\rm dR}(M)/{\rm Fil}^0 H_{\rm dR}(M).$$
Since Hodge-Tate weights of $V^\ast(1)$ at $\fp$ (resp. $\barfp$) are positive (resp. non-positive), we have
$$F\otimes_\cF t(M) \simeq D_{{\rm dR},\fp}(V^\ast(1)).$$
Recall that $M$ is critical, i.e., the period map is an isomorphism:
$$\alpha_M: \RR\otimes_\QQ H_{B}(M)^+ \xrightarrow{\sim} \RR\otimes_\QQ t(M). $$
Take $\cF$-bases $\gamma \in \bigwedge_\cF^2 H_B(M)^+$ and $\delta \in \bigwedge_\cF^2 t(M)$. Let
\begin{equation}\label{period C}
\alpha_{M,\CC}: \CC \otimes_\cF {\bigwedge}_\cF^2 H_B(M)^+ \xrightarrow{\sim} \CC\otimes_\cF {\bigwedge}_\cF^2 t(M)
\end{equation}
be the isomorphism induced by $\alpha_M$. We define the period 
\begin{equation}\label{complex period}
\Omega_{\gamma,\delta} \in \CC^\times
\end{equation}
with respect to $\gamma$ and $\delta$ by
$$\alpha_{M,\CC}(\gamma)=\Omega_{\gamma,\delta}\cdot \delta.$$
Deligne's conjecture \cite{deligne} for the motive $M$ (see Conjecture \ref{deligne0}) states that
$$\frac{L(f,\chi^{-1},r)}{\Omega_{\gamma,\delta}} \in \cF,$$
which is known to be true (essentially due to Shimura \cite{shimura}, see \cite[Thm. 5.5]{BDP} and Lemma \ref{compare cm1} below). 

We now state the Tamagawa number conjecture for the pair $(M,\cO_F)$. 
We fix a stable $\cO_F$-lattice $T\subset V$ and set $T^\ast(1):=\Hom_{\cO_F}(T,\cO_F(1))$. Take an $\cF$-basis $\gamma \in \bigwedge_\cF^2 H_B(M)^+$ so that its image under the comparison isomorphism
$$F\otimes_\cF {\bigwedge}_\cF^2 H_B(M)^+\simeq {\bigwedge}_F^2 V^\ast(1)$$
is an $\cO_F$-basis of ${\bigwedge}_{\cO_F}^2 T^\ast(1)$. If we assume $H^1_f(K,V)=H^1_f(K,V^\ast(1))=0$, then by Corollary \ref{cor:fp} we have a canonical identification
$${\det}_F^{-1}(\rgamma(G_{K,S},V)) = {\bigwedge}_F^2 H^1(G_{K,S},V)$$
and the localization isomorphism
$${\rm loc}_\fp : {\bigwedge}_F^2 H^1(G_{K,S},V) \xrightarrow{\sim} {\bigwedge}_F^2 H^1(K_\fp,V).$$
Also, the Bloch-Kato dual exponential map induces an isomorphism
$$\exp_\fp^\ast: {\bigwedge}_F^2 H^1(K_\fp,V) \xrightarrow{\sim} {\bigwedge}_F^2 D_{{\rm dR},\fp}(V) .$$
Via the isomorphism $F\otimes_\cF t(M) \simeq D_{{\rm dR},\fp}(V^\ast(1))$, we regard the $\cF$-basis $\delta \in {\bigwedge}_\cF^2 t(M)$ as an $F$-basis of ${\bigwedge}_F^2 D_{{\rm dR},\fp}(V^\ast(1))$. Let $\delta^\ast \in {\bigwedge}_F^2 D_{{\rm dR},\fp}(V^\ast(1))^\ast \simeq {\bigwedge}_F^2 D_{{\rm dR},\fp}(V)$ denote the dual basis. 

\begin{conjecture}[The Tamagawa number conjecture for $(M,\cO_F)$ in analytic rank zero]\label{TNC}
Assume $L(f,\chi^{-1},r)\neq 0$. Then we have $H^1_f(K,V)=H^1_f(K,V^\ast(1))=0$, and there is an $\cO_F$-basis
$$\fz_\gamma \in {\det}_{\cO_F}^{-1}(\rgamma(G_{K,S},T))$$
such that the composition map
$${\det}_F^{-1}(\rgamma(G_{K,S},V)) = {\bigwedge}_F^2 H^1(G_{K,S},V) \stackrel{{\rm loc}_\fp}{\simeq} {\bigwedge}_F^2 H^1(K_\fp,V)\stackrel{\exp^\ast_\fp}{\simeq}{\bigwedge}_F^2 D_{{\rm dR},\fp}(V)$$
sends $\fz_\gamma$ to
$$\frac{L_S(f,\chi^{-1},r)}{\Omega_{\gamma,\delta}}\cdot \delta^\ast.$$
Here $L_S(f,\chi^{-1},s)$ denotes the $L$-series obtained by removing the Euler factors at $v\in S$ from $L(f,\chi^{-1},s)$. 
\end{conjecture}


\begin{remark}
The validity of Conjecture \ref{TNC} is independent of the choices of $S$, $T$, $\gamma$ and $\delta$. 
\end{remark}

\subsection{The Iwasawa main conjecture}\label{sec IMC}

We review the formulation of the Iwasawa main conjecture for the Bertolini-Darmon-Prasanna (BDP) $p$-adic $L$-function and state the main results of this article. 

Let $F_f$ be the minimal extension of $\QQ_p$ which contains the Fourier coefficients of $f$ via $\iota_p: \overline \QQ \hookrightarrow \overline \QQ_p$. Let $\cO_f:=\cO_{F_f}$ be the ring of integers of $F_f$. Let $V_f$ be the $p$-adic Galois representation attached to $f$ with coefficients in $F_f$. Fix a stable $\cO_f$-lattice $T_f \subset V_f$. 

Let $K_\infty/K$ be the anticyclotomic $\ZZ_p$-extension and set $\Gamma:=\Gal(K_\infty/K)$. 
Let $\widehat \ZZ_p^{\rm ur}$ be the completion of the ring of integers of the maximal unramified extension $\QQ_p^{\rm ur}$ of $\QQ_p$. We set
$$\Lambda:=\cO_{f}[[\Gamma]] \text{ and }\Lambda^{\rm ur}:= \widehat \ZZ_p^{\rm ur}\widehat \otimes \Lambda.$$


The BDP $p$-adic $L$-function for $f$
$$L_\fp^{\rm BDP} \in \Lambda^{\rm ur}$$
satisfies the following interpolation property: for a Hecke character of infinity type $(j,-j)$ with $j\geq r$ such that its $p$-adic avatar $\chi$ factors through $\Gamma$ and is crystalline at both $\fp$ and $\barfp$, we have
\begin{equation}\label{BDPinterpolation}
\chi(L_\fp^{\rm BDP}) = \Omega_p^{4j}\cdot \Gamma(j-r+1) \Gamma(j+r)\cdot  (1-a_p \chi^{-1}(\barfp) p^{-r} + \chi^{-2}(\barfp) p^{-1})^2 \frac{L(f,\chi^{-1},r)}{ \Omega_\infty^{4j} (2\pi)^{1-2j}\sqrt{D_K}^{2j-1} }
\end{equation}
Here $\Omega_\infty \in \CC^\times$ and $\Omega_p \in (\widehat \ZZ_p^{\rm ur})^\times$ denote CM periods, whose definitions are given in \S \ref{sec CM} below. (Note that $\Omega_K$ in \cite[\S 2.5]{CH} satisfies $\Omega_\infty= 2\pi i \cdot \Omega_K$. Our $L_\fp^{\rm BDP}$ is the involution of $\mathscr{L}_\fp(f,{\bf 1})^2$, where $\mathscr{L}_\fp(f,{\bf 1})$ is as in \cite[Thm. 2.1.3]{castellaTNC}.) 
See \cite{BDP}, \cite{brakocevic}, \cite{CH} for the construction of $L_\fp^{\rm BDP}$. 

Let 
$$\TT:= \Lambda \otimes_{\cO_f}T_f(r)$$
be the deformation of $T_f(r)$, on which $G_K$ acts by
$$\sigma \cdot (x\otimes y) := \overline \sigma^{-1} x\otimes \sigma y \quad (\sigma \in G_K, \ x\in \Lambda, \ y\in T_f(r)),$$
where $\overline \sigma \in \Gamma$ denotes the image of $\sigma$. 
We consider the following Selmer complex:
\begin{equation}\label{def selmer}
\widetilde \rgamma_\fp(K,\TT):= {\rm Cone}\left(\rgamma(G_{K,S}, \TT) \to \rgamma(K_\fp, \TT) \oplus \bigoplus_{v\in S, v\nmid p}\rgamma_{/{\rm ur}}(K_v,\TT)\right)[-1].
\end{equation}
Here we set
$$\rgamma_{/{\rm ur}}(K_v,\TT):= {\rm Cone}\left(\rgamma_{\rm ur}(K_v,\TT)\xrightarrow{-{\rm Inf}_{K_v^{\rm ur}/K_v}} \rgamma(K_v,\TT)\right).$$
Note that $\widetilde \rgamma_\fp(K,\TT)$ coincides with the Selmer complex defined in \cite[\S 6.1]{nekovar} for the local condition
$$U_v^+(\TT):= \begin{cases}
0 &\text{if $v=\fp$,}\\
\rgamma(K_{\barfp},\TT) &\text{if $v=\barfp$,}\\
\rgamma_{\rm ur}(K_v,\TT) &\text{if $v\nmid p$.}
\end{cases}$$

For a commutative ring $R$, let $Q(R)$ be the total quotient ring of $R$. The Iwasawa main conjecture for the BDP $p$-adic $L$-function is stated as follows. 

\begin{conjecture}[The Iwasawa main conjecture]\label{IMC}
The complex $Q(\Lambda) \lotimes_{\Lambda} \widetilde \rgamma_\fp(K,\TT)$ is acyclic and there is a $\Lambda^{\rm ur}$-basis
$$\fz_\fp \in \Lambda^{\rm ur} \otimes_{\Lambda} {\det}_{\Lambda}^{-1}(\widetilde \rgamma_\fp(K,\TT))$$
such that
$$\pi(\fz_\fp)= L_\fp^{\rm BDP},$$
where $\pi$ denotes the canonical isomorphism
$$\pi: Q(\Lambda^{\rm ur})\otimes_{\Lambda} {\det}_{\Lambda}^{-1}(\widetilde \rgamma_\fp(K,\TT)) \xrightarrow{\sim} Q(\Lambda^{\rm ur}).$$
\end{conjecture}

\begin{remark}
We set $\bA:= \TT^\vee(1)$ and 
$$H^1_{\barfp}(K,\bA) := \ker \left(H^1(G_{K,S},\bA) \to H^1(K_{\barfp},\bA)\oplus \bigoplus_{v\in S, v\nmid p}H^1_{/{\rm ur}}(K_v,\bA)\right).$$
(Note that $H^1_{/{\rm ur}}(K_v,\bA)=H^1(K_v,\bA)$ if $v\nmid p$ by \cite[Lem. B.3.3]{R}.) 
Let $\iota: \Lambda^{\rm ur}\to \Lambda^{\rm ur}$ be the involution induced by $\Gamma\to \Gamma; \ \gamma \mapsto \gamma^{-1}$. 
The usual formulation of the Iwasawa main conjecture is the following (see \cite[Conj. 4.1]{sanoderived}): {\it $H^1_{\barfp}(K,\bA)^\vee$ is $\Lambda$-torsion and }
$$\Lambda^{\rm ur}\cdot {\rm char}_\Lambda(H^1_{\barfp}(K,\bA)^\vee) = \Lambda^{\rm ur}\cdot \iota(L_\fp^{\rm BDP}).$$
This is equivalent to Conjecture \ref{IMC} by \cite[Prop. 4.5]{sanoderived}.
\end{remark}

The following is the main result of this article. 

\begin{theorem}\label{thm:main}
Let $\chi$ be an anticyclotomic Hecke character of $K$ such that its $p$-adic avatar factors through $\Gamma$ and is crystalline at both $\fp$ and $\barfp$. Assume $L(f,\chi^{-1},r)\neq 0$. 
Then Conjecture \ref{IMC} implies Conjecture \ref{TNC}. 
\end{theorem}

\begin{remark}
As noted in Remark \ref{rem CHK}, we have $H^1_f(K,V)=H^1_f(K,V^\ast(1))=0$ if $L(f,\chi^{-1},r)\neq 0$. (Note that, since $\chi$ is anticyclotomic, we have $\chi^{-1}(v)=\chi(\overline v)$ for any finite place $v$ and so $L(f,\chi^{-1},r)\neq 0$ is equivalent to $L(f,\chi,r)\neq 0$.) Hence it is sufficient to prove the last claim of Conjecture \ref{TNC}. 
\end{remark}

\begin{corollary}\label{cor:main}
Let $\chi$ be as in Theorem \ref{thm:main}. 
Suppose $r=1$ and $f$ corresponds to an elliptic curve $E$. Assume the following:
\begin{itemize}
\item $L(f,\chi^{-1},1)\neq 0$;
\item the representation $\rho:G_\QQ\to {\rm Aut}(E[p])$ is surjective;
\item $\rho$ is ramified at every $\ell | N$;
\item $N$ is square-free;
\item $p$ is non-anomalous.
\end{itemize}
Then the Tamagawa number conjecture for $(h^1(E/K)(1)\otimes M(\chi), \cO_F)$ is true. 
\end{corollary}

\begin{proof}
This follows from \cite[Thm. B]{BCK} (ordinary) and \cite[Cor. 7.2]{CHKLL} (supersingular), where Conjecture \ref{IMC} is proved under the stated assumptions. 
\end{proof}

\section{Proof}

In this section, we give a proof of Theorem \ref{thm:main}. 

As explained in \S \ref{sec idea}, the key idea is to construct a $\Lambda$-basis
$$\fz_S \in {\det}_\Lambda^{-1}(\rgamma(G_{K,S},\TT))$$
which is related with the value $L(f,\chi^{-1},r)$. We first give a review on the local epsilon element in \S \ref{sec local epsilon}, which is necessary in the construction of $\fz_S$. In \S \ref{sec construction}, we construct $\fz_S$ by using $\fz_\fp$ in the Iwasawa main conjecture (Conjecture \ref{IMC}) and the local epsilon element. (Note that in this construction we do not need to assume $K$ has class number one.) Theorem \ref{thm:main} is then reduced to the interpolation property of $\fz_S$, which is Theorem \ref{thm:interpolation}. We prove Theorem \ref{thm:interpolation} by comparing various periods. We give preliminaries on CM periods and choice of bases of Betti/de Rham cohomology in \S\S \ref{sec CM} and \ref{sec choice}. We complete the proof of Theorem \ref{thm:interpolation} (and hence Theorem \ref{thm:main}) in \S \ref{sec compare}.

\subsection{The local epsilon element}\label{sec local epsilon}

We first take a natural basis of ${\bigwedge}_\Lambda^2 \TT= \Lambda \otimes_{\cO_f} {\bigwedge}_{\cO_f}^2 T_f(r)$ in the following way. Since we regard $\overline \QQ \subset \CC$, we have a canonical $p^n$-th root of unity $\zeta_{p^n}:=e^{2\pi i/p^n} \in \overline\QQ$. The collection of these gives a $\ZZ_p$-basis $\xi := (\iota_p(\zeta_{p^n}))_n \in H^0(\overline \QQ_p ,\ZZ_p(1))$. Since we have a canonical isomorphism ${\bigwedge}_{\cO_f}^2 T_f(r)\simeq \cO_f (1)$, we can regard $\xi$ as a $\Lambda$-basis of $\bigwedge_{\Lambda}^2 \TT$, which we denote by $\gamma_\TT$.

Kato's local epsilon conjecture (see \cite{katolecture2} or \cite[Conj. 3.4.3]{FK}) predicts the existence of a $\Lambda^{\rm ur}$-basis 
$$\varepsilon_{\Lambda,\xi}(\TT) \in \Lambda^{\rm ur} \otimes_\Lambda {\det}_\Lambda(\rgamma(K_\fp,\TT)) \otimes_\Lambda {\bigwedge}_\Lambda^2 \TT$$
(``local epsilon element", see Remark \ref{rem isom basis} below) satisfying certain interpolation properties. In our setting, this conjecture is known, thanks to work of Loeffler-Venjakob-Zerbes \cite{LVZ}, Nakamura \cite{nakamura} and Rodrigues~Jacinto \cite{RJ}. In fact, since $f$ is good at $p$, $\TT=\Lambda\otimes_{\cO_f}T_f(r)$ is a deformation of a crystalline representation of $G_{\QQ_p}$, which is treated in \cite{LVZ}. (Note that $K_\fp\simeq \QQ_p$.) Alternatively, since $T_f(r)$ is a representation of $G_{\QQ_p}$ of rank two, the validity of the local epsilon conjecture is covered by \cite{nakamura}, \cite{RJ}. 

\begin{remark}\label{rem isom basis}
In \cite{FK}, a ``local epsilon isomorphism"
$$\varepsilon_{\Lambda,\xi}(\TT):{\rm Det}_{\Lambda^{\rm ur}}(0) \xrightarrow{\sim} \Lambda^{\rm ur} \otimes_\Lambda ( {\rm Det}_\Lambda(\rgamma(K_\fp,\TT)) \cdot {\rm Det}_\Lambda (\TT) )$$
is considered, where ${\rm Det}_\Lambda$ denotes the determinant functor as in \cite[\S 1.2]{FK}. In our commutative setting, we replace ${\rm Det}_\Lambda$ with the determinant module ${\rm det}_\Lambda$ and identify the isomorphism $\varepsilon_{\Lambda,\xi}(\TT)$ with the image of $1 \in \Lambda^{\rm ur}= {\det}_{\Lambda^{\rm ur}}(0)$, which is a basis of $\Lambda^{\rm ur} \otimes_\Lambda ( {\det}_\Lambda(\rgamma(K_\fp,\TT)) \otimes_\Lambda {\det}_\Lambda (\TT) )$. 
\end{remark}

Define
\begin{equation}\label{def local epsilon}
\varepsilon_\fp \in \Lambda^{\rm ur}\otimes_\Lambda {\det}_\Lambda(\rgamma(K_\fp,\TT))
\end{equation}
to be the $\Lambda^{\rm ur}$-basis such that 
$$\varepsilon_\fp\otimes\gamma_\TT = \varepsilon_{\Lambda,\xi}(\TT).$$

\begin{proposition}\label{phiepsilon}
Let $\varphi_p$ be the arithmetic Frobenius acting on $\widehat \ZZ_p^{\rm ur}$ (and hence on $\Lambda^{\rm ur}$). Let $\chi_{\rm cyc}: G_{\QQ_p}\to \ZZ_p^\times$ be the cyclotomic character. Let $\QQ_p^{\rm ab}$ be the maximal abelian extension of $\QQ_p$ and $\tau_p\in G_{\QQ_p}^{\rm ab}:=\Gal(\QQ_p^{\rm ab}/\QQ_p)$ the unique lift of the arithmetic Frobenius such that $\chi_{\rm cyc}(\tau_p)=1$. Let $\overline \tau_p \in \Gamma$ be the image of $\tau_p$. Then we have
$$\varphi_p(\varepsilon_\fp)= \overline \tau_p^2 \cdot \varepsilon_\fp.$$
\end{proposition}

\begin{proof}
This follows from the property in \cite[Conj. 3.4.3(iv)]{FK} by noting that $\tau_p$ acts on ${\bigwedge}_\Lambda^2 \TT$ via multiplication by $\overline \tau_p^{-2} \in \Gamma$. 
\end{proof}

Recall $V:=V_f(r)\otimes_{F_f}F(\chi^{-1})$. 
Let $\widehat \QQ_p^{\rm ur}$ be the completion of the maximal unramified extension of $\QQ_p$ and set 
$$\widetilde F:=\widehat \QQ_p^{\rm ur}\otimes_{\QQ_p}F.$$ 
We set $D_{\rm dR}(V)=D_{{\rm dR},\fp}(V):=H^0(K_\fp, B_{\rm dR}\otimes_{\QQ_p} V)$. We recall the definition of the isomorphism
\begin{equation}\label{dr isom}
\varepsilon_{\rm dR}=\varepsilon_{F,\xi,{\rm dR}}(V): \widetilde F \otimes_F {\bigwedge}_F^2 D_{{\rm dR}}(V)\xrightarrow{\sim} \widetilde F \otimes_F {\bigwedge}_F^2 V
\end{equation}
constructed in \cite[\S 3.3.4]{FK}. 
Note that Hodge-Tate weights of $V$ (as a $G_{K_\fp}$-representation) are $r-j$ and $1-r-j$. Hence we have
$$m:=\sum_{i\in \ZZ}i \dim_F({\rm gr}^i D_{{\rm dR}}(V))= (-r+j)+(-1+r+j) =2j-1.$$
Let
$${\rm can}: (B_{\rm dR}\otimes_{\QQ_p} F) \otimes_F {\bigwedge}_F^2 D_{{\rm dR}}(V) \xrightarrow{\sim} (B_{\rm dR}\otimes_{\QQ_p}F) \otimes_F {\bigwedge}_F^2 V$$
be the canonical isomorphism. We define 
$$\varepsilon_{\rm dR}:= t_\xi^{-m} \cdot {\rm can},$$
where $t_\xi \in B_{\rm dR}^+$ denotes the uniformizer corresponding to $\xi$. (Since $V$ is crystalline, the linearized action of the Weil group on $D_{\rm pst}(V)$ is unramified (see \cite[Prop. 2.3.2]{LVZ}), and hence the epsilon constant $\varepsilon_F(D_{\rm pst}(V),\psi)$ (as in \cite[\S 3.3.4]{FK}) is $1$ by \cite[(4) in \S 3.2.2]{FK}.) One can check that $\varepsilon_{\rm dR}$ has coefficients in $\widetilde F$ (see \cite[Prop. 3.3.5]{FK}).  

\begin{proposition}\label{epsilon twist}
Let
$$\varepsilon_\fp^\chi \in \widetilde F \otimes_{F} {\det}_F(\rgamma(K_\fp,V))= \Hom_{ F}\left( {\bigwedge}_F^2H^1(K_\fp,V), \widetilde F\right).$$
be the image of $\varepsilon_\fp$ under the $\chi$-twisting map:
$$\Lambda^{\rm ur}\otimes_\Lambda {\det}_\Lambda(\rgamma(K_\fp,\TT)) \xrightarrow{a\mapsto a\otimes 1} \Lambda^{\rm ur}\otimes_\Lambda {\det}_\Lambda(\rgamma(K_\fp,\TT))\otimes_{\Lambda,\chi}F \simeq \widetilde F \otimes_{F} {\det}_F(\rgamma(K_\fp,V)).
$$
(The last isomorphism follows from \cite[Prop. 1.6.5(3)]{FK}. Note that $V:=V_f(r)\otimes_{F_f}F(\chi^{-1})\simeq \TT\otimes_{\Lambda,\chi} F$.)
Then $\varepsilon_\fp^\chi$ coincides with the composition of the following maps:
\begin{itemize}
\item[(a)] the map induced by the Bloch-Kato dual exponential map
$$\exp^\ast_\fp: {\bigwedge}_F^2 H^1(K_\fp,V)\xrightarrow{\sim} {\bigwedge}_F^2 D_{{\rm dR}}(V) = {\bigwedge}_F^2 D_{\rm cris}(V),$$
\item[(b)] the automorphism $(1-\varphi)(1-p^{-1}\varphi^{-1})^{-1}$ on ${\bigwedge}_F^2 D_{\rm cris}(V)$ (which coincides with multiplication by $(1-a_p \chi(\fp) p^{-r} + \chi^{2}(\fp)p^{-1})(1-a_p \chi^{-1}(\fp) p^{-r} + \chi^{-2}(\fp)p^{-1})^{-1}$),
\item[(c)] the isomorphism
$$\varepsilon_{\rm dR}: \widetilde F \otimes_F {\bigwedge}_F^2 D_{{\rm dR}}(V)\xrightarrow{\sim} \widetilde F \otimes_F {\bigwedge}_F^2 V \simeq \widetilde F,$$
where the last isomorphism is determined by the basis $\gamma_T\in {\bigwedge}_F^2 V$ obtained from the fixed basis $\gamma_\TT\in {\bigwedge}_\Lambda^2 \TT$ via $\TT \otimes_{\Lambda,\chi}F\simeq V$, 
\item[(d)] multiplication by $-\Gamma(j-r+1)\Gamma(j+r)$.
\end{itemize}
\end{proposition}

\begin{proof}
This follows from the properties in \cite[Conj. 3.4.3(ii) and (v)]{FK} and the definition of $\varepsilon_{F,\xi}(V)$ in \cite[\S 3.3]{FK}. Note that the composition of the maps in (a) and (b) is essentially $\theta_F(V)$ in \cite[\S 3.3.2]{FK}. Note that $\Gamma_F(V)$ in \cite[\S 3.3.6]{FK} is equal to $-\Gamma(j-r+1)\Gamma(j+r)$. 
\end{proof}

\subsection{Construction of a basis}\label{sec construction}

In the following, we assume the Iwasawa main conjecture (Conjecture \ref{IMC}).

By the definition of $\widetilde \rgamma_\fp(K,\TT)$ (see (\ref{def selmer})), we have a canonical isomorphism
\begin{equation}\label{can det}
{\det}_\Lambda^{-1}( \rgamma(G_{K,S}, \TT) ) \otimes_\Lambda {\det}_\Lambda( \rgamma(K_\fp, \TT)  ) \otimes_\Lambda \bigotimes_{v\in S, v\nmid p}{\det}_\Lambda(\rgamma_{/{\rm ur}}(K_v,\TT)) \simeq {\det}_\Lambda^{-1}(  \widetilde \rgamma_\fp(K,\TT))    .
\end{equation}

The following is well-known. 

\begin{lemma}\label{lem euler}
Let $v\in S$ be a finite place such that $v\nmid p $. Then ${\det}_\Lambda(\rgamma_{/{\rm ur}}(K_v,\TT))$ has a canonical $\Lambda$-basis whose image under the map
$${\det}_\Lambda(\rgamma_{/{\rm ur}}(K_v,\TT)) \xrightarrow{a \mapsto a\otimes 1} {\det}_\Lambda(\rgamma_{/{\rm ur}}(K_v,\TT)) \otimes_{\Lambda,\chi} F \simeq {\det}_F(\rgamma_{/{\rm ur}}(K_v,V)) \simeq F$$
is the Euler factor $\det(1-{\rm Fr}_v^{-1}\mid V^\ast(1)^{I_v})^{-1}$, where ${\rm Fr}_v$ denotes the arithmetic Frobenius and $I_v \subset G_{K_v}$ is the inertia subgroup. (The last isomorphism is due to the fact that $\rgamma_{/{\rm ur}}(K_v,V)$ is acyclic.)
\end{lemma}

\begin{proof}
Note that the complex $\rgamma_{\rm ur}(K_v,\TT)$ is represented by
$$\left[ \TT^{I_v} \xrightarrow{1-{\rm Fr}_v^{-1}} \TT^{I_v}\right].$$
Hence we have ${\det}_\Lambda^{-1}(\rgamma_{\rm ur}(K_v,\TT)) = {\det}_\Lambda^{-1}(\TT^{I_v})\otimes_\Lambda {\det}_\Lambda(\TT^{I_v})$ and it has a canonical basis. (In fact, if we take any $\Lambda$-basis $t\in {\det}_\Lambda(\TT^{I_v})$, then the element $t^\ast \otimes t$ is independent of $t$.) By duality, we have a canonical isomorphism 
$${\det}_\Lambda^{-1}(\rgamma_{\rm ur}(K_v,\TT)) \simeq {\det}_\Lambda(\rgamma_{/{\rm ur}}(K_v,\TT)),$$
and we get a canonical basis of ${\det}_\Lambda(\rgamma_{/{\rm ur}}(K_v,\TT))$. (We identify $T_f(r)^\ast(1)$ with $T_f(r)$.) By construction, it has the desired property. 
\end{proof}

By (\ref{can det}) and Lemma \ref{lem euler}, we obtain a canonical isomorphism
$$\left(\Lambda^{\rm ur}\otimes_\Lambda {\det}_\Lambda^{-1}(\rgamma(G_{K,S},\TT)) \right)\otimes_{\Lambda^{\rm ur}} \left(\Lambda^{\rm ur}\otimes_\Lambda {\det}_\Lambda(\rgamma(K_\fp, \TT)) \right)\simeq \Lambda^{\rm ur}\otimes_\Lambda {\det}_\Lambda^{-1}(\widetilde \rgamma_\fp(K,\TT)).$$
Assuming the Iwasawa main conjecture (Conjecture \ref{IMC}), we define a $\Lambda^{\rm ur}$-basis
$$\fz_S \in \Lambda^{\rm ur}\otimes_\Lambda {\det}_\Lambda^{-1}(\rgamma(G_{K,S},\TT)) $$
to be the element such that $\fz_S \otimes \varepsilon_\fp $ corresponds to $\fz_\fp$ under the isomorphism above. 

\begin{proposition}\label{coeff}
$\fz_S$ lies in ${\det}_\Lambda^{-1}(\rgamma(G_{K,S},\TT)) $. 
\end{proposition}

\begin{proof}
Let $\varphi_p$ be the arithmetic Frobenius acting on the coefficient $\widehat \ZZ_p^{\rm ur}$ of $\Lambda^{\rm ur}$. It is sufficient to show that $\varphi_p(\fz_S)=\fz_S$. By Proposition \ref{phiepsilon}, the claim is reduced to Lemma \ref{phiBDP} below. 
\end{proof}

\begin{lemma}\label{phiBDP}
Let $\varphi_p$ and $\overline \tau_p$ be as in Proposition \ref{phiepsilon}. Then we have
$$\varphi_p(L_\fp^{\rm BDP})= \overline \tau_p^2 \cdot L_\fp^{\rm BDP}.$$
\end{lemma}

\begin{proof}
When $p$ is ordinary, this follows from the Coleman-type construction of the BDP $p$-adic $L$-function due to Castella-Hsieh \cite[Thm. 5.7]{CH} and the property of the Coleman-Perrin-Riou regulator map in \cite[Prop. 4.9]{LZ}. When $p$ is supersingular and $f$ corresponds to an elliptic curve, the claim follows from \cite[Thm. 6.2]{castellawan}, since $\Xi_d$ in loc. cit., regarded as an element of $\Lambda^{\rm ur}$, satisfies $\varphi_p(\Xi_d)=\overline \tau_p\cdot \Xi_d$. 

In the general case, we can prove the claim by the following argument. Let $L_\fp^{\rm Katz} \in \widehat \ZZ_p^{\rm ur}[[\Gamma]]$ be the anticyclotomic projection of the Katz $p$-adic $L$-function. (By our convention, we let this to be the involution of $\mathscr{L}_\fp^{\rm ac}(K)$ in \cite{castella}.) Since  $L_\fp^{\rm Katz}$ is the image of a system of elliptic units under the Coleman map (see \cite{yager}), we see that
\begin{equation}\label{katz}
\varphi_p(L_\fp^{\rm Katz}) = \overline \tau_p\cdot L_\fp^{\rm Katz}. 
\end{equation}
by \cite[Prop. 4.9]{LZ}. 
On the other hand, by the argument of \cite[Thm. 1.7]{castella} (see also \cite[\S 5.3]{JSW}), we have
\begin{equation}\label{BDP Katz}
L_\fp^{\rm BDP} = \cL \cdot q(L_\fp^{\rm Katz})
\end{equation}
for some $\cL \in \Lambda$, where $q: \Lambda^{\rm ur}\to \Lambda^{\rm ur}$ denotes the map induced by $\gamma \mapsto \gamma^2$ for $\gamma \in \Gamma$. (Note that \cite[Thm. 1.7]{castella} holds even when $p$ is supersingular and the weight of $f$ is greater than $2$. Note also that, since $\Gamma$ is the Galois group of the anticyclotomic $\ZZ_p$-extension, we have $\gamma^2 = \gamma^{1-\rho}$, where $\rho \in \Gal(K/\QQ)$ denotes the complex conjugation.) The claim follows immediately from (\ref{katz}) and (\ref{BDP Katz}). 
\end{proof}

By Proposition \ref{coeff}, we can define an $\cO_F$-basis
$$\fz_S^\chi \in {\det}_{\cO_F}^{-1}(\rgamma(G_{K,S},T))$$
to be the image of $\fz_S$ under the $\chi$-twisting map
$${\det}_\Lambda^{-1}(\rgamma(G_{K,S},\TT)) \xrightarrow{a \mapsto a\otimes 1}{\det}_\Lambda^{-1}(\rgamma(G_{K,S},\TT)) \otimes_{\Lambda,\chi} \cO_F \simeq {\det}_{\cO_F}^{-1}(\rgamma(G_{K,S},T)).$$
(The last isomorphism is due to \cite[Prop. 1.6.5(3)]{FK}.)

Recall that in \S \ref{sec TNC} we took an $\cO_F$-basis $\gamma \in {\bigwedge}_{\cO_F}^2 T^\ast(1)$ and an $F$-basis $\delta^\ast \in {\bigwedge}_F^2 D_{{\rm dR}}(V^\ast(1))^\ast \simeq {\bigwedge}_F^2 D_{{\rm dR}}(V)$. Let $\Omega_{\gamma,\delta} \in \CC^\times$ be the period with respect to these bases (see (\ref{complex period})). 
Theorem \ref{thm:main} is now reduced to the following.

\begin{theorem}\label{thm:interpolation}
The composition map
$$\lambda_\fp: {\det}_F^{-1}(\rgamma(G_{K,S},V)) = {\bigwedge}_F^2 H^1(G_{K,S},V) \stackrel{{\rm loc}_\fp}{\simeq} {\bigwedge}_F^2 H^1(K_\fp,V)\stackrel{\exp^\ast_\fp}{\simeq}{\bigwedge}_F^2 D_{{\rm dR}}(V)$$
sends $\fz_S^\chi$ to
$$\frac{L_S(f,\chi^{-1},r)}{\Omega_{\gamma,\delta}}\cdot \delta^\ast$$
up to a unit in $\cO_F^\times$. 
\end{theorem}

The rest of this section is devoted to the proof of Theorem \ref{thm:interpolation}. 

\subsection{CM periods}\label{sec CM}
We review the definitions of CM periods $\Omega_\infty \in \CC^\times$ and $\Omega_p \in (\widehat \ZZ_p^{\rm ur})^\times$. 

Let $A$ be the canonical elliptic curve defined over $K$ with complex multiplication by $\cO_K$ as in \cite[Thm. 0.1]{yang}. (Note that $K$ has class number one and $D_K$ is odd. Also, $A$ descends to an elliptic curve  over $\QQ$.) Fix a global minimal Weierstrass model of $A$ over $\cO_K$ and let $\omega_A \in \Gamma(A,\Omega_{A/K}^1)$ be the corresponding N\'eron differential. We also fix an $\cO_K$-basis $\gamma_A \in H_1(A(\CC),\ZZ)$. We define the complex CM period by
$$\Omega_\infty:=\int_{\gamma_A}\omega_A.$$

Next, we define the $p$-adic CM period. Let $\widehat A$ be the formal group of $A$ over $\cO_{K_\fp}\simeq \ZZ_p$ with respect to the parameter $-x/y$ (as in \cite[p.47]{deshalit}). Let $T_\fp(A)$ be the $\fp$-adic Tate module of $A$. Then we have a canonical isomorphism $H_1(A(\CC),\ZZ)\otimes_{\cO_K}\cO_{K_\fp} \simeq T_\fp(A)$. Also, note that the $p$-adic Tate module $T_p(\widehat A)$ of $\widehat A$ is identified with $T_\fp(A)$ as $G_{K_\fp}$-representations. Thus we can regard $\gamma_A \in H_1(A(\CC),\ZZ)$ as a $\ZZ_p$-basis of $T_p(\widehat A)$, which we denote by $\gamma_{A,p}$. Let
$$\eta_A: \widehat \GG_m \xrightarrow{\sim} \widehat A $$
be the isomorphism of formal groups over $\widehat \ZZ_p^{\rm ur}$ which corresponds (by \cite{tate}) to the isomorphism
$$\ZZ_p(1)=T_p(\widehat \GG_m)\xrightarrow{\sim} T_p(\widehat A); \ \xi \mapsto \gamma_{A,p}.$$
Regarding $\eta_A \in \widehat \ZZ_p^{\rm ur}[[X]]$, we define the $p$-adic CM period by
$$\Omega_p:= \eta_A'(0).$$
In other words, we have
$$\eta_A^\ast (\omega_A) = \Omega_p\frac{dX}{1+X}.$$

For later purpose, we shall give another description of $\Omega_p$. 

\begin{proposition}\label{prop Omega}
Let
$$\Gamma(A,\Omega_{A/K}^1) \times T_p(\widehat A) \to B_{\rm dR}; \ (\omega,\gamma)\mapsto \int_\gamma \omega$$
be the $p$-adic integration constructed by Colmez \cite[Prop. 3.1]{colmez}. Then we have
\begin{equation*}\label{padic period int}
\Omega_p= t_\xi^{-1} \int_{\gamma_{A,p}}\omega_A,
\end{equation*}
where $t_\xi \in B_{\rm dR}^+$ denotes the uniformizer corresponding to $\xi \in \ZZ_p(1)$. 
(Note that the right hand side actually lies in $\widehat \QQ_p^{\rm ur}= H^0(\QQ_p^{\rm ur},B_{\rm dR})$, since $T_p(\widehat A)\simeq \ZZ_p(1)$ as $G_{\QQ_p^{\rm ur}}$-representations and so $\sigma \in G_{\QQ_p^{\rm ur}}$ acts on $\int_{\gamma_{A,p}}\omega_A$ by $\chi_{\rm cyc}(\sigma)$.) 
\end{proposition}

\begin{proof}
We first recall the construction of the $p$-adic integration. Let $\cO_{\CC_p^\flat}:=\varprojlim \cO_{\CC_p}$ be the ``tilt" of the ring of integers $\cO_{\CC_p} $ of $\CC_p$, where the inverse limit is taken with respect to the $p$-th power map. For $x \in \cO_{\CC_p^\flat}$, we write $x^{(0)} \in \cO_{\CC_p}$ for the $0$-th component of $x$. Let $W(\cO_{\CC_p^\flat})$ be the ring of Witt vectors of $\cO_{\CC_p^\flat}$. Then there is a natural surjective homomorphism
$$\theta: W(\cO_{\CC_p^\flat}) \to \cO_{\CC_p}; \ \sum_{n=0}^\infty [x_n]p^n \mapsto \sum_{n=0}^\infty x_n^{(0)}p^n,$$
where for $x \in \cO_{\CC_p^\flat}$ we write $[x] \in W(\cO_{\CC_p^\flat})$ for its Teichm\"uller representative. Let $A_{\rm inf}$ be the completion of $W(\cO_{\CC_p^\flat})$ for the topology defined by $(p)+\ker \theta$. The induced homomorphism $A_{\rm inf}[1/p]\to \CC_p$ is also denoted by $\theta$. We write $\gamma_{A,p}=(\gamma_n)_n \in \varprojlim_n \widehat A[p^n]=T_p(\widehat A)$ and take $\widetilde \gamma_n \in A_{\rm inf}$ such that $\theta(\widetilde \gamma_n)=\gamma_n$ for each $n$. Let $\log_{\widehat A} \in \QQ_p[[X]]$ be the formal logarithm of $\widehat A$. Then by definition we have
$$\int_{\gamma_{A,p}}\omega_A:= \underset{n\to \infty}{\lim} p^n \log_{\widehat A}(\widetilde \gamma_n).$$
(This converges in $B_{\rm dR}^+=\varprojlim_n A_{\rm inf}[1/p]/(\ker \theta)^n$ and is independent of the choice of each $\widetilde \gamma_n$. See \cite[Prop. 3.1(i)]{colmez}. Note also that we ignore the sign: it is not important for our purpose.)  

Next, we recall the definition of $t_\xi \in B_{\rm dR}$. 
We can naturally regard $\xi \in \cO_{\CC_p^\flat}$ and so we can consider its Teichm\"uller representative $[\xi] \in A_{\rm inf}$. We define $t_\xi := \log_{\widehat \GG_m}([\xi]-1)$, where $\log_{\widehat \GG_m}(X):=\sum_{n=1}^\infty (-1)^{n-1} \frac{X^n}{n} $. 

To prove the proposition, we make a specific choice of $\widetilde \gamma_n$. 
We write $\xi=(\xi_n)_n \in \varprojlim_n \mu_{p^n} = \ZZ_p(1)$. By our choice of the isomorphism $\eta_A: \widehat \GG_m\xrightarrow{\sim} \widehat A$, we have $\eta_A(\xi_n-1)=\gamma_n$. Since we have $\theta([\xi^{p^{-n}}])= \xi_n$, the element $\widetilde \gamma_n := \eta_A([\xi^{p^{-n}}]-1)$ satisfies $\theta(\widetilde \gamma_n)=\gamma_n$. Using this element, we have
$$\int_{\gamma_{A,p}}\omega_A = \underset{n\to \infty}{\lim} p^n \log_{\widehat A}(\widetilde \gamma_n) = \log_{\widehat A}(\eta_A([\xi]-1)).$$
Since $\log_{\widehat A}\circ \eta_A = \Omega_p \log_{\widehat \GG_m}$ by the definition of $\Omega_p$, we have
$$\log_{\widehat A}(\eta_A([\xi]-1)) = \Omega_p \log_{\widehat \GG_m} ([\xi]-1) = \Omega_p t_\xi.$$
This completes the proof. 
\end{proof}

\subsection{Choice of bases}\label{sec choice}

In order to compare various periods, we make specific choices of $\gamma \in {\bigwedge}_{\cO_F}^2 T^\ast(1)$ and $\delta^\ast \in {\bigwedge}_F^2 D_{{\rm dR}}(V)$.

Let $\psi =\psi_A$ be the Hecke character associated with the CM elliptic curve $A$ fixed in \S \ref{sec CM}. We first need the following.

\begin{lemma}\label{lem hecke}
Let $\chi: G_K \to \overline \QQ_p^\times$ be a character which is crystalline at $\fp$  and corresponds to a Hecke character of infinity type $(j,-j)$. Let $K_\infty/K$ be the anticyclotomic $\ZZ_p$-extension. Assume that $\chi$ factors through $\Gamma=\Gal(K_\infty/K)$. Then we have $p-1\mid j$ and 
$$\chi=\psi^j\overline \psi^{-j}.$$
\end{lemma}

\begin{proof}
Let $K_\fp^{\rm ur}$ be the maximal unramified extension of $K_\fp$. 
Let $I_\fp:=G_{K_\fp^{\rm ur}} \subset G_{K_\fp}$ be the inertia subgroup of $\fp$. Since $\chi$ is crystalline at $\fp$, we have $\chi=\chi_{\rm cyc}^j$ on $I_\fp$ (see \cite[Cor. 9.3.2]{BC}). Since $\chi$ factors through $\Gamma$, we see that $\chi_{\rm cyc}^j$ factors through $\Gal(K_{\fp,\infty}^{\rm ur}/K_\fp^{\rm ur})$, where $K_{\fp,\infty}^{\rm ur}:=K_\infty K_\fp^{\rm ur}$. Since $K_{\fp,\infty}^{\rm ur}$ coincides with the cyclotomic $\ZZ_p$-extension of $K_\fp^{\rm ur}$, we have $p-1 \mid j$.

We set $\chi_0:= \chi \psi^{-j}\overline \psi^j$. We show that $\chi_0$ is the trivial character. 
Let $K_\psi$ be the field corresponding to the kernel of $\psi$. Then $K_\psi/K$ is totally ramified at $\fp$, unramified outside $\fp D_K$ and $\Gal(K_\psi/K)\simeq \ZZ_p^\times$ (see \cite[II.1.7 and II.1.9]{deshalit}). 
Since $p-1\mid j$, we see that $\psi^j$ factors through $\Gal(K(\fp)_\infty/K)$, where $K(\fp)_\infty$ denotes the unique $\ZZ_p$-extension of $K$ unramified outside $\fp$. Similarly, $\overline \psi^j$ factors through $\Gal(K(\barfp)_\infty/K)$. Hence we see that $\psi^{-j} \overline \psi^j$ factors through $\Gamma$. This implies that $\chi_0$ also factors through $\Gamma$. Since $\chi_0$ is a finite order character which is crystalline at $\fp$, its restriction on $I_\fp$ is trivial. Since $K_\infty/K$ is totally ramified at $\fp$, we see that $\chi_0$ is trivial. Hence we have completed the proof. 
\end{proof}

By Lemma \ref{lem hecke}, we can write $\chi= \psi^j\overline \psi^{-j}$. Since $\psi\overline \psi$ is the norm Hecke character $\N$, we can also write $\chi= \psi^{2j}\N^{-j} $. Note that the $p$-adic avatar of $\N$ is the cyclotomic character $\chi_{\rm cyc}$. 

We choose $\gamma \in {\bigwedge}_{\cO_F}^2 T^\ast(1)$ in the following way. 
Note that we have an isomorphism
$${\bigwedge}_{\cO_F}^2 T^\ast(1) \simeq \ZZ_p(1)\otimes_{\ZZ_p}\cO_F( \chi^2) = T_\fp(A)^{\otimes 4j} \otimes_{\ZZ_p} \cO_F(1-2j).$$
Let $\gamma_{A,p} \in T_\fp(A)=T_p(\widehat A)$ be the basis chosen in \S \ref{sec CM}. We take $\gamma \in {\bigwedge}_{\cO_F}^2 T^\ast(1)$ to be the element corresponding to $\gamma_{A,p}^{\otimes 4j} \otimes \xi^{\otimes (1-2j)} $ under the isomorphism above. 

Next, we choose $\delta^\ast \in {\bigwedge}_F^2 D_{\rm dR}(V)$. We similarly have
\begin{align}\label{ddr isom}
{\bigwedge}_F^2 D_{\rm dR}(V) = D_{\rm dR}\left({\bigwedge}_F^2 V\right) &\simeq D_{\rm dR}\left(V_\fp(A)^{\otimes (-4j)} \otimes_{\QQ_p} F(2j+1)\right) \\
&= D_{\rm dR}(V_\fp(A)^{\otimes (-4j)}) \otimes_{\QQ_p} D_{\rm dR}(F(2j+1)) ,\nonumber
\end{align}
where we set $V_\fp(A):=\QQ_p\otimes_{\ZZ_p} T_\fp(A)$. Let
$$\omega_{A,p}\in D_{\rm dR}(V_\fp(A)^\ast)$$
be the image of $1\otimes \omega_A$ under the canonical isomorphism $K_\fp \otimes_K \Gamma(A,\Omega_{A/K}^1)\simeq D_{\rm dR}(V_\fp(A)^\ast)$. Explicitly, we have
$$\omega_{A,p} = t_\xi \Omega_p \otimes \gamma_{A,p}^\ast$$
by Proposition \ref{prop Omega}. We set
$$e_k := t_\xi^{-k}\otimes \xi^k \in D_{\rm dR}(F(k))$$
for any $k\in \ZZ$. 
We define $\delta^\ast \in {\bigwedge}_F^2 D_{\rm dR}(V)$ to be the element corresponding to 
$$\omega_{A,p}^{\otimes 4j} \otimes e_{2j+1}  \in D_{\rm dR}(V_\fp(A)^{\otimes (-4j)}) \otimes_{\QQ_p} D_{\rm dR}(F(2j+1)) $$
under the isomorphism (\ref{ddr isom}). Explicitly, we have
$$\delta^\ast = t_\xi^{2j-1} \Omega_p^{4j}\otimes \gamma_{A,p}^{\otimes (-4j)} \otimes \xi^{\otimes (2j+1)}  $$
as an element of $B_{\rm dR}\otimes_{\QQ_p} V_\fp(A)^{\otimes (-4j)} \otimes_{\QQ_p} F(2j+1)$. 

\begin{remark}\label{rem basis}
Under the comparison isomorphisms
$$F\otimes_{\cO_K, \iota_p}H_1(A(\CC),\ZZ)^{\otimes 4j}(1-2j)\simeq V_\fp(A)^{\otimes 4j}\otimes_{\QQ_p} F(1-2j) \simeq {\bigwedge}_F^2 V^\ast(1),$$
$$F\otimes_{K,\iota_p}\Gamma(A,\Omega_{A/K}^1)^{\otimes (-4j)} \simeq D_{\rm dR}\left(V_\fp(A)^{\otimes 4j} \otimes_{\QQ_p} F(1-2j)\right) \simeq {\bigwedge}_F^2 D_{\rm dR}(V^\ast(1)),$$
the bases $\gamma$ and $\delta$ correspond to $(2\pi i)^{1-2j}\gamma_A^{\otimes 4j}$ and $\omega_A^{\otimes(-4j)}$ respectively. 
\end{remark}

\subsection{Comparison of periods}\label{sec compare}

Let $\gamma_T \in {\bigwedge}_{\cO_F}^2 T$ be the $\cO_F$-basis obtained from the fixed basis $\gamma_\TT\in {\bigwedge}_\Lambda^2 \TT$ via $\TT \otimes_{\Lambda,\chi}\cO_F\simeq T$. We define a $p$-adic period
$$\Omega_{p, \gamma_T,\delta^\ast} \in \widetilde F^\times$$
by
$$\varepsilon_{\rm dR}(\delta^\ast)= \Omega_{p,\gamma_T,\delta^\ast}\cdot \gamma_T,$$
where $\varepsilon_{\rm dR}: \widetilde F \otimes_F {\bigwedge}_F^2 D_{{\rm dR}}(V)\xrightarrow{\sim} \widetilde F \otimes_F {\bigwedge}_F^2 V$ is the isomorphism defined in (\ref{dr isom}). We regard $\Omega_{p,\gamma_T,\delta^\ast} \in \widehat F^{\rm ur}$ via the natural map $\widetilde F = \widehat \QQ_p^{\rm ur}\otimes_{\QQ_p}F \to \widehat F^{\rm ur}$, where $\widehat F^{\rm ur}$ denotes the completion of the maximal unramified extension of $F$. 

In the following, we write
$$a\sim b$$
if the equality $a=b$ holds up to a unit in $\cO_F^\times$. 

\begin{lemma}\label{compare cm padic}
We have
$$\Omega_{p,\gamma_T,\delta^\ast} \sim \Omega_p^{4j}.$$
\end{lemma}

\begin{proof}
Note that $\gamma_T \sim \gamma_{A,p}^{\otimes (-4j)} \otimes \xi^{\otimes (2j+1)} $ if we identify ${\bigwedge}_{\cO_F}^2 T $ with $T_\fp(A)^{\otimes (-4j)} \otimes_{\ZZ_p} \cO_F(2j+1) $. 
The canonical isomorphism
$${\rm can}: B_{\rm dR}\otimes_F {\bigwedge}_F^2 D_{\rm dR}(V) = B_{\rm dR} \otimes_F D_{\rm dR}\left({\bigwedge}_F^2 V \right)\xrightarrow{\sim} B_{\rm dR}\otimes_F {\bigwedge}_F^2 V$$
sends 
$$1\otimes \delta^\ast= 1\otimes ( t_\xi^{2j-1} \Omega_p^{4j}\otimes \gamma_{A,p}^{\otimes (-4j)} \otimes \xi^{\otimes (2j+1)}  )$$
to 
$$ t_\xi^{2j-1} \Omega_p^{4j}\otimes \gamma_{A,p}^{\otimes (-4j)} \otimes \xi^{\otimes (2j+1)} \otimes 1  \sim t_\xi^{2j-1}\Omega_p^{4j}\otimes \gamma_T.$$
Since we have $\varepsilon_{\rm dR}=t_\xi^{1-2j}\cdot {\rm can}$ by definition, we see that
$$\varepsilon_{\rm dR}(\delta^\ast) = \Omega_p^{4j}\cdot \gamma_T. $$
This proves the claim. 
\end{proof}

\begin{lemma}\label{compare cm1}
We have
$$\Omega_{\gamma,\delta} =\pm \Omega_\infty^{4j} (2\pi)^{1-2j}\sqrt{D_K}^{2j-1}. $$
\end{lemma}

\begin{proof}
Note that the period map (\ref{period C}) coincides with the period map for the motive ${\bigwedge}^2 M = h^0(\Spec \QQ)(1)\otimes M(\chi^2)$ (see Remark \ref{rem det modular}). This motive coincides with the Hecke motive associated with ${\N} \chi^2=\psi^{2j+1}\overline \psi^{-2j+1}$, where $\N$ denotes the norm Hecke character. 

In general, it is known that the period (with respect to the natural bases determined by $\gamma_A$ and $\omega_A$) of the critical motive associated with a Hecke character $\psi^k \overline \psi^\ell$ with $\ell \leq  0 < k$ is 
$$\pm \Omega_\infty^{k-\ell} (2\pi)^{\ell}\sqrt{D_K}^{-\ell}$$
(see Proposition \ref{hecke period}). When $k=2j+1$ and $\ell=-2j+1$, this is  
$$\pm \Omega_\infty^{4j}(2\pi)^{1-2j}\sqrt{D_K}^{2j-1}.$$
Since we made natural choices of $\gamma$ and $\delta$ as in Remark \ref{rem basis}, this coincides with $\Omega_{\gamma,\delta}$. 
\end{proof}

We now give a proof of Theorem \ref{thm:interpolation}. 

\begin{proof}[Proof of Theorem \ref{thm:interpolation}]

Let $\lambda_\fp$ be the composition map in Theorem \ref{thm:interpolation}. 
By Proposition \ref{epsilon twist} and Lemma \ref{lem euler}, we have
\begin{multline*}
-\Gamma(j-r+1)\Gamma(j+r)\cdot (1-a_p \chi^{-1}(\barfp) p^{-r} + \chi^{-2}(\barfp)p^{-1})(1-a_p \chi^{-1}(\fp) p^{-r} + \chi^{-2}(\fp)p^{-1})^{-1}\cdot {\rm Eul}^{-1}\cdot \varepsilon_{\rm dR}(\lambda_\fp(\fz_S^\chi)) \\
= \chi(L_\fp^{\rm BDP})\cdot \gamma_T.
\end{multline*}
(Note that $\chi(\fp)=\chi^{-1}(\barfp)$ since $\chi$ is anticyclotomic.) Here ${\rm Eul}$ is the product of Euler factors at $v\in S\setminus \{\fp,\barfp\}$, which satisfies ${\rm Eul}\cdot L_{\{\fp,\barfp\}}(f,\chi^{-1},r)= L_S(f,\chi^{-1},r)$. (Explicitly, we have ${\rm Eul}=\prod_{v\in S, v\nmid p}\det(1-{\rm Fr}_v^{-1}\mid V^\ast(1)^{I_v})$.) 
By the formula (\ref{BDPinterpolation}), we obtain
$$\varepsilon_{\rm dR}(\lambda_\fp(\fz_S^\chi))\sim  \Omega_p^{4j}\cdot \frac{L_S(f,\chi^{-1},r)}{\Omega_\infty^{4j} (2\pi)^{1-2j}\sqrt{D_K}^{2j-1}}\cdot \gamma_T.$$
From this, we have
\begin{equation*}\label{z formula}
\lambda_\fp(\fz_S^\chi)\sim \frac{\Omega_p^{4j}}{\Omega_{p,\gamma_T,\delta^\ast}}\cdot \frac{L_S(f,\chi^{-1},r)}{\Omega_\infty^{4j} (2\pi)^{1-2j}\sqrt{D_K}^{2j-1}}\cdot \delta^\ast.
\end{equation*}
Finally, by Lemmas \ref{compare cm padic} and \ref{compare cm1}, we obtain
$$\lambda_\fp(\fz_S^\chi)\sim \frac{L_S(f,\chi^{-1},r)}{\Omega_{\gamma,\delta}}\cdot \delta^\ast.$$
This completes the proof of Theorem \ref{thm:interpolation}. 
\end{proof}

\appendix

\section{Motives} \label{sec motive}

In this article, we regard a ``motive" as a collection of realizations and comparison isomorphisms. This is sufficient in order to formulate the Tamagawa number conjecture. In this appendix, we clarify the definition of motives in our sense. (Our treatment of motives is similar to \cite[\S 6]{fontaineL}, \cite[Chap. III]{FP}, where the category of ``motivic structures" is considered.) We also review the formulation of the Tamagawa number conjecture in the critical case. In the last section we give basic examples of motives. 

Throughout this appendix, let $\cF$ and $K$ be number fields. Let $S_p(\cF)$ be the set of $p$-adic places of $\cF$. 

\subsection{Motivic structures}

Let $d,w\in \ZZ$ with $d\geq 0$. A (pure) motive $M$ defined over $K$ (of rank $d$ and weight $w$) with coefficients in $\cF$ (denoted by $(M,\cF)$) is the following collection of realizations and comparison isomorphisms. The realizations are the following. 
\begin{itemize}
\item For each embedding $\sigma: K\hookrightarrow \CC$, we have the $\sigma$-Betti realization
$$H_\sigma(M),$$
which is an $\cF$-vector space of dimension $d$ endowed with a (pure) $\QQ$-Hodge structure of weight $w=w(M)$. For $c \in G_\RR=\Gal(\CC/\RR)$, there is an isomorphism
$$H_\sigma(M)\xrightarrow{\sim}H_{c\circ \sigma}(M),$$
which is also denoted by $c$. 
The Betti realization of $M$ is defined by
$$H_B(M):=\bigoplus_{\sigma:K\hookrightarrow \CC}H_\sigma(M).$$
The action of $c \in G_\RR$ on $H_B(M)$ is defined by $c\cdot (a_\sigma)_\sigma:=(c\cdot a_\sigma)_{c\circ \sigma}$. 
\item For each $\lambda \in S_p(\cF)$, we have the $\lambda$-adic realization
$$V_\lambda(M),$$
which is an $\cF_\lambda$-vector space of dimension $d$ endowed with a continuous $\cF_\lambda$-linear $G_K$-action, unramified outside a finite set $S$ of places of $K$. The $p$-adic realization of $M$ is defined by
$$V_p(M):=\bigoplus_{\lambda \in S_p(\cF)} V_\lambda(M).$$
\item We have the de Rham realization
$$H_{\rm dR}(M),$$
which is a free $\cF\otimes_\QQ K$-module of rank $d$ endowed with a decreasing filtration $ \{ {\rm Fil}^iH_{\rm dR}(M)\}_{i\in \ZZ}$ such that ${\rm Fil}^i H_{\rm dR}(M)=H_{\rm dR}(M)$ for $i \ll 0$ and ${\rm Fil}^i H_{\rm dR}(M)=0$ for $i \gg 0$. The tangent space of $M$ is defined by
$$t(M):=H_{\rm dR}(M)/{\rm Fil}^0H_{\rm dR}(M).$$
\end{itemize}
These realizations have the following comparison isomorphisms. 
\begin{itemize}
\item For each $\sigma:K\hookrightarrow \CC$, we have
$$\CC\otimes_\QQ H_\sigma(M) \simeq \CC\otimes_{K,\sigma}H_{\rm dR}(M)$$
which is compatible with the Hodge structure. 
This induces an isomorphism
$$\CC\otimes_\QQ H_B(M)\simeq \CC\otimes_\QQ H_{\rm dR}(M).$$
\item For each $\sigma: K\hookrightarrow \CC$ and $\lambda \in S_p(\cF)$, we have
$$\cF_\lambda \otimes_\cF H_\sigma(M)\simeq V_\lambda(M).$$
This induces an isomorphism
$$\QQ_p\otimes_\QQ H_\sigma(M)\simeq V_p(M).$$
\item For each $\sigma: K\hookrightarrow \overline \QQ_p$, we have
$$B_{\rm dR}\otimes_{K,\sigma} H_{\rm dR}(M)\simeq B_{\rm dR}\otimes_{\QQ_p} V_p(M)$$
which is compatible with filtration. 
For each $p$-adic place $\fp$ of $K$, this induces an isomorphism
$$K_\fp \otimes_K H_{\rm dR}(M)\simeq D_{{\rm dR},\fp}(V_p(M)):= H^0(K_\fp, B_{\rm dR}\otimes_{\QQ_p}V_p(M)).$$
\end{itemize}

\begin{remark}\label{twist}
For $j\in \ZZ$, we can consider the Tate twist $M(j)$ of $M$. Its realizations are given as follows. 
\begin{itemize}
\item $H_\sigma(M(j)):= H_\sigma(M)\otimes_\QQ (2\pi i)^j \QQ$.
\item $V_\lambda(M(j)):=V_\lambda(M)(j) (:=V_\lambda(M)\otimes_{\QQ_p} \QQ_p(j))$. 
\item $H_{\rm dR}(M(j)):= H_{\rm dR}(M)$ with filtration ${\rm Fil}^i H_{\rm dR}(M(j)):={\rm Fil}^{i+j}H_{\rm dR}(M)$. 
\end{itemize}
Similarly, we can consider the dual $M^\ast$ of $M$, whose realizations are the following. 
\begin{itemize}
\item $H_\sigma(M^\ast):=\Hom_\cF(H_B(M), \cF)$.
\item $V_\lambda(M^\ast):= \Hom_{\cF_\lambda}(V_\lambda(M), \cF_\lambda)$. 
\item $H_{\rm dR}(M^\ast):= \Hom_K(H_{\rm dR}(M),K)$ with filtration 
$${\rm Fil}^i H_{\rm dR}(M^\ast):= \Hom_K(H_{\rm dR}(M)/ {\rm Fil}^{1-i} H_{\rm dR}(M), K).$$ 
\end{itemize}
In particular, we can consider the ``Kummer dual" $M^\ast(1)$ of $M$. Note that
$$t(M^\ast(1))= \Hom_K({\rm Fil}^0 H_{\rm dR}(M),K).$$
One can also define the tensor product $M\otimes N$ and the set of homomorphisms $\Hom(M,N)$ for two  motives $M$ and $N$. 
\end{remark}

\begin{remark}\label{base extension}
We have the following observation concerning change of bases and coefficients (see \cite[\S 6.4]{fontaineL}). 

Let $M$ be a motive defined over $\QQ$. Then we can regard $M$ as a motive defined over $K$. If we denote this motive by $M_K$, its realizations are given as follows. 
\begin{itemize}
\item For each $\sigma: K\hookrightarrow \CC$, $H_\sigma(M_K):=H_B(M)$. 
\item For each $\lambda \in S_p(\cF)$, $V_\lambda(M_K):=V_\lambda(M)$ (regarded as a $G_K$-representation). 
\item $H_{\rm dR}(M_K):=K\otimes_\QQ H_{\rm dR}(M)$. 
\end{itemize}
Conversely, for a motive $M$ defined over $K$, one can define a motive ${\rm Res}_{K/\QQ}(M)$ defined over $\QQ$ by the Weil restriction. Similarly, for an extension of coefficient fields $\cF'/\cF$, one can regard $(M,\cF)$ as $(M,\cF')$ by the scalar extension, and conversely, $(M,\cF')$ as $(M,\cF)$ by the restriction. 
\end{remark}

We assume that $M$ satisfies the following hypothesis. 
Fix a prime number $p$ and $\lambda \in S_p(\cF)$. 
For a finite place $v$ of $K$, let ${\rm Fr}_v$ denote the arithmetic Frobenius of $v$ and $I_v \subset G_{K_v}$ the inertia subgroup of $v$. If $v\mid p$, we set $D_{{\rm cris},v}(-):=H^0(K_v,B_{\rm cris}\otimes_{\QQ_p}-)$ and we write $\varphi$ for the Frobenius acting on it. We set
$$P_v(M,x):= \begin{cases}
\det(1-{\rm Fr}_v^{-1} x \mid V_\lambda(M)^{I_v}) &\text{if $v\nmid p$,}\\
\det(1-\varphi x \mid D_{{\rm cris},v}(V_\lambda(M))) &\text{if $v\mid p$.}
\end{cases}$$
Note that $P_v(M,x) \in \cF_\lambda[x]$. We say that a finite place $v$ of $K$ is good if $V_\lambda(M)$ is unramified at $v$ (resp. crystalline at $v$) when $v\nmid p$ (resp. $v\mid p$). 

\begin{hypothesis}\label{conj L}\
\begin{itemize}
\item[(i)] We have $P_v(M,x)\in \cF[x]$ for any finite place $v$ of $K$. 
\item[(ii)] Let $w=w(M) \in \ZZ$ be the weight of $M$. Then for any good place $v$ of $K$ and any root $\alpha\in \CC$ of $P_v(M,x)$, we have $|\alpha|= {\N}v^{-w/2}$. (Here ${\N}v$ denotes the cardinality of the residue field of $v$.)
\item[(iii)] The $L$-function of $M$
$$L(M,s):= \prod_v P_v(M,{\N}v^{-s})^{-1},$$
where $v$ runs over all finite places of $K$, is analytically continued to $s=0$. 
\end{itemize}
\end{hypothesis}

\begin{remark}
Since we regard $\cF\subset \CC$ by our convention, we regard $L(M,s)$ as a $\CC$-valued function. (Without fixing an embedding $\cF\hookrightarrow \CC$, $L(M,s)$ is regarded as a $\CC\otimes_\QQ \cF$-valued function.)
\end{remark}

Finally, we define a critical motive. 
We set 
$$H_B(M)^+:=H^0(\RR, H_B(M)).$$
The period map for $M$
$$\alpha_M: \RR\otimes_\QQ H_B(M)^+\to \RR\otimes_\QQ t(M)$$
is defined to be the map induced by the comparison isomorphism $\CC\otimes_\QQ H_B(M)\simeq \CC\otimes_\QQ H_{\rm dR}(M)$. 

\begin{definition}
$M$ is said to be critical if $\alpha_M$ is an isomorphism. 
\end{definition}

\begin{remark}\label{rem critical}
In \cite[Def. 1.3]{deligne}, it is defined that $M$ is critical if neither $L_\infty(M,s)$ nor $L_\infty(M^\ast,1-s)$ has a pole at $s=0$. Here $L_\infty(M,s)$ denotes the ``$\Gamma$-factor" (or ``$L$-factor at infinity") of $M$ determined by the Hodge structure of $H_B(M)$ (see \cite[\S 5.2]{deligne}). One checks that this is equivalent to our definition. 
Also, one sees that $M$ is critical if $w(M)=-1$. 
\end{remark}

\subsection{Motivic cohomology}

Let $(M,\cF)$ be a motive defined over $K$ with coefficients in $\cF$. We have the following standard conjectures. 

\begin{conjecture}[{See \cite[Conj. 2]{BFetnc}}]\label{conj motivic}
For each $i\in \{0,1\}$, one can define a motivic cohomology group
$$H^i_f(K,M),$$
which is a finite dimensional $\cF$-vector space, and a canonical isomorphism
$$\cF_\lambda \otimes_\cF H^i_f(K,M) \simeq H^i_f(K,V_\lambda(M))$$
for any $\lambda \in S_p(\cF)$. 
\end{conjecture}

\begin{conjecture}\label{conj height}
Assume that $M$ is critical. Then there is a non-degenerate height pairing
$$h_M: H^1_f(K,M)\times H^1_f(K,M^\ast(1)) \to \RR\otimes_\QQ \cF. $$
\end{conjecture}

\begin{remark}\label{rem height}
More generally, it is conjectured that there is a canonical exact sequence
\begin{multline*}
0\to \RR\otimes_\QQ H^0_f(K,M) \to \ker \alpha_M \to \RR\otimes_\QQ H^1_f(K,M^\ast(1))^\ast  \\
\xrightarrow{h_M} \RR\otimes_\QQ H_f^1(K,M)\to \coker \alpha_M \to \RR\otimes_\QQ H_f^0(K,M^\ast(1))^\ast\to 0.
\end{multline*}
(See \cite[Conj. 1]{BFetnc}.) Note that, if $M$ is critical (i.e., $\ker \alpha_M=\coker \alpha_M=0$), this exact sequence implies not only the non-degeneracy of the height pairing but also $H^0_f(K,M)=H^0_f(K,M^\ast(1))=0$. 
\end{remark}

\subsection{The Deligne conjecture}

We review the Deligne conjecture \cite{deligne}, which is a special case of the ``rationality" part of the Tamagawa number conjecture (see \cite[Conj. 4(iii)]{BFetnc}). 

Assume that $M$ is critical. Then the period map is an isomorphism:
$$\alpha_M: \RR\otimes_\QQ H_B(M)^+ \xrightarrow{\sim} \RR\otimes_\QQ t(M). $$
Let 
$$\det(\alpha_M) : \RR\otimes_\QQ {\det}_\cF(H_B(M)^+) \xrightarrow{\sim} \RR\otimes_\QQ {\det}_\cF(t(M))$$
be the induced isomorphism. Take $\cF$-bases
$$\gamma \in {\det}_\cF(H_B(M)^+) \text{ and }\delta \in {\det}_\cF(t(M)).$$
We define the period 
$$\Omega_{\gamma,\delta} \in (\RR\otimes_\QQ\cF)^\times$$
with respect to $\gamma$ and $\delta$ by
$${\det}(\alpha_M)(\gamma)=\Omega_{\gamma,\delta}\cdot \delta.$$ 
We regard
$$\Omega_{\gamma,\delta}\in \CC^\times$$
via the map $\RR\otimes_\QQ \cF\to \CC; \ a\otimes b\mapsto ab$. 

\begin{conjecture}[The Deligne conjecture]\label{deligne0}
Assume that $M$ is critical. Then we have
$$\frac{L(M,0)}{\Omega_{\gamma,\delta}} \in \cF.$$
\end{conjecture}

Note that the validity of Conjecture \ref{deligne0} is obviously independent of the choices of $\gamma$ and $\delta$. Note also that the conjecture is trivial when $L(M,0)=0$. 

We shall state a general version of the Deligne conjecture, which treats the case $L(M,0)=0$. Let
$$L^\ast(M,0):=\underset{s\to 0}{\lim} s^{-\ord_{s=0}L(M,s)} L(M,s)$$
be the leading term of $L(M,s)$ at $s=0$. Assuming Conjecture \ref{conj height}, let 
$$\det(h_M): \RR\otimes_\QQ \left( {\det}_\cF(H^1_f(K,M)) \otimes_\cF {\det}_\cF(H^1_f(K,M^\ast(1)))\right)\xrightarrow{\sim}\RR\otimes_\QQ \cF$$
be the isomorphism induced by the height pairing $h_M$. Take $\cF$-bases
$$x \in {\det}_\cF(H^1_f(K,M)) \text{ and }y \in {\det}_\cF(H^1_f(K,M^\ast(1)))$$
and define the regulator with respect to $x$ and $y$ by
$$R_{x,y}:= {\det}(h_M)(x\otimes y) \in (\RR\otimes_\QQ \cF)^\times.$$
We regard
$$R_{x,y}\in \CC^\times$$
via the map $\RR\otimes_\QQ \cF\to \CC; \ a\otimes b\mapsto ab$. 

\begin{conjecture}[The generalized Deligne conjecture]\label{deligne general}
Assume $M$ is critical and Conjecture \ref{conj height}. Then we have
$$\frac{L^\ast(M,0)}{\Omega_{\gamma,\delta} R_{x,y}} \in \cF.$$
\end{conjecture}

The validity of Conjecture \ref{deligne general} is independent of the choices of $\gamma,\delta,x,y$.

\subsection{The Tamagawa number conjecture in the critical case}

We review the formulation of the Tamagawa number conjecture. 

The following is \cite[Conj. 4(ii)]{BFetnc}.

\begin{conjecture}[The ``order of vanishing"]\label{conj order}
Assume the existence of motivic cohomology groups $H^i_f(K,M^\ast(1))$. Then we have
$$\ord_{s=0} L(M,s)= \dim_\cF(H^1_f(K,M^\ast(1))) -\dim_\cF(H^0_f(K,M^\ast(1))). $$
\end{conjecture}

\begin{remark}\label{rem conj order}
Without assuming the existence of motivic cohomology groups, we can conjecture that
\begin{equation}\label{orderBK}
\ord_{s=0} L(M,s)= \dim_{F}(H^1_f(K,V)) -\dim_{F}(H^0(K,V)),
\end{equation}
where $F:=\cF_\lambda$ and $V:=V_\lambda(M^\ast(1))$ with $\lambda \in S_p(\cF)$. 
This is equivalent to Conjecture \ref{conj order} if we assume Conjecture \ref{conj motivic}. We remark that (\ref{orderBK}) is often referred to simply as the ``Bloch-Kato conjecture" in the literature (for example, \cite{CH}). 
\end{remark}

We shall formulate the ``integrality" part of the Tamagawa number conjecture (see \cite[Conj. 4(iv)]{BFetnc}). We only treat the critical case, since we do not consider non-critical cases in this article.

Fix an odd prime number $p$ and $\lambda \in S_p(\cF)$. We set
$$F:=\cF_\lambda \text{ and }V:=V_\lambda(M^\ast(1)).$$ 
Note that $V_\lambda(M)=V^\ast(1)$. Take a finite set $S$ of places of $K$ containing all the infinite places, $p$-adic places, and the places at which $V$ ramify. Let $G_{K,S}$ be the Galois group of the maximal extension of $K$ unramified outside $S$. We fix a $G_K$-stable $\cO_F$-lattice $T\subset V$. 
Recall that we set
$$P_v(M,x):= \begin{cases}
\det(1-{\rm Fr}_v^{-1} x \mid V^\ast(1)^{I_v}) &\text{if $v\nmid p$,}\\
\det(1-\varphi x \mid D_{{\rm cris},v}(V^\ast(1))) &\text{if $v\mid p$.}
\end{cases}$$

We formulate the Tamagawa number conjecture under the following hypothesis. 

\begin{hypothesis}\label{TNChyp}\
\begin{itemize}
\item[(i)] $M$ is critical.
\item[(ii)] $P_v(M,1)\neq 0$ for every finite place $v$ of $K$. 
\item[(iii)] $H^0(K,V)=0$. 
\end{itemize}
\end{hypothesis}

\begin{remark}
Hypothesis \ref{TNChyp}(ii) is equivalent to the following: $H^0(K_v,V^\ast(1))=0$ for $v\nmid p$ and $D_{{\rm cris},v}(V^\ast(1))^{\varphi=1}=0$ for $v\mid p$. Also, by the Bloch-Kato fundamental exact sequence
$$0\to \QQ_p\to B_{\rm cris}^{\varphi=1}\to B_{\rm dR}/B_{\rm dR}^+\to 0,$$
we see that $D_{{\rm cris},v}(V^\ast(1))^{\varphi=1}=0$ implies $H^0(K_v,V^\ast(1))=0$ for $v\mid p$. 
\end{remark}

%


\begin{remark}
According to Conjecture \ref{conj motivic} and Remark \ref{rem height}, Hypothesis \ref{TNChyp}(iii) should always be satisfied if $M$ is critical. 
\end{remark}

We use the Poitou-Tate exact sequence
\begin{multline}\label{TNCPT}
0\to H^1_f(K,V)\to H^1(G_{K,S},V)\to \bigoplus_{v\in S}H^1_{/f}(K_v,V)\\
\to H^1_f(K,V^\ast(1))^\ast \to H^2(G_{K,S},V)\to \bigoplus_{v\in S_f}H^0(K_v,V^\ast(1))^\ast.
\end{multline}
(Here $S_f\subset S$ denotes the subset of finite places.) Note that the last term vanishes by Hypothesis \ref{TNChyp}(ii). 

We first consider the case $H^1_f(K,V)=H^1_f(K,V^\ast(1))=0$. (According to Conjecture \ref{conj order}, this is the case when the ``analytic rank" is zero, i.e., $L(M,0)\neq 0$.) In this case, (\ref{TNCPT}) implies that $H^2(G_{K,S},V)=0$ and that the localization maps at $p$-adic places induce an isomorphism
$${\rm loc}_p: H^1(G_{K,S},V) \xrightarrow{\sim} \bigoplus_{v \in S_p(K)}H^1_{/f}(K_v,V), $$
where $S_p(K)$ denotes the set of $p$-adic places of $K$. (Note that Hypothesis \ref{TNChyp}(ii) implies $H^1_{/f}(K_v,V)\simeq H^1_f(K_v,V^\ast(1))^\ast=0$ for $v\nmid p$.) Also, since Hypothesis \ref{TNChyp}(ii) implies $D_{{\rm cris},v}(V^\ast(1))^{\varphi=1}=0$ for $v\in S_p(K)$, the dual exponential map induces an isomorphism
$$\exp_v^\ast: H^1_{/f}(K_v,V)\xrightarrow{\sim} D_{{\rm dR},v}^0(V).$$
Note that we have a canonical isomorphism
$$D_{{\rm dR},v}^0(V)\simeq \left(D_{{\rm dR},v}(V^\ast(1))/D^0_{{\rm dR},v}(V^\ast(1))\right)^\ast.$$
Combining this with the comparison isomorphism
$$F\otimes_\cF t(M)\simeq \bigoplus_{v\in S_p(K)} D_{{\rm dR},v}(V^\ast(1))/D_{{\rm dR},v}^0(V^\ast(1)),$$
we can regard $\bigoplus_{v\in S_p(K)}\exp_v^\ast$ as an isomorphism
$$\exp^\ast: \bigoplus_{v\in S_p(K)}H^1_{/f}(K_v,V)\xrightarrow{\sim} F\otimes_\cF t(M)^\ast. $$
For an $\cF$-basis $\delta \in {\det}_\cF(t(M))$, let $\delta^\ast \in {\det}_\cF(t(M)^\ast)\simeq {\det}_\cF(t(M))^\ast$ denote the dual basis. 

Note that by $H^2(G_{K,S},V)=0$ and Hypothesis \ref{TNChyp}(iii) we have an identification
$${\det}_F^{-1}(\rgamma(G_{K,S},V))= {\det}_F(H^1(G_{K,S}, V)) .$$
Consider the composition map
$$\vartheta_0: {\det}_F^{-1}(\rgamma(G_{K,S},V))= {\det}_F(H^1(G_{K,S}, V)) \stackrel{{\rm loc}_p}{\simeq} {\det}_F\left(\bigoplus_{v\in S_p(K)} H^1_{/f}(K_v,V)\right) \stackrel{{\rm exp^\ast}}{\simeq} F\otimes_\cF {\det}_\cF(t(M)^\ast).$$

We set
$$L_S(M,s):= \left(\prod_{v\in S}P_v(M, {\N}v^{-s})\right) L(M,s).$$
Note that the comparison isomorphism $F\otimes_\cF H_B(M)\simeq \bigoplus_{\sigma:K\hookrightarrow \CC} V^\ast(1)$ induces an isomorphism
$$F\otimes_\cF H_B(M)^+\simeq \bigoplus_{v\in S_\infty(K)} H^0(K_v, V^\ast(1)),$$
where $S_\infty(K)$ denotes the set of infinite places of $K$. 

\begin{conjecture}[The Tamagawa number conjecture for $(M,\cO_F)$ in analytic rank zero]\label{TNC general0}
Assume Hypothesis \ref{TNChyp} and $H^1_f(K,V)=H^1_f(K,V^\ast(1))=0$. Assume also that $L(M,0)\neq 0$ and the Deligne conjecture (Conjecture \ref{deligne0}) is true. Let $\gamma \in {\det}_\cF(H_B(M)^+)$ be an $\cF$-basis such that its image under the comparison isomorphism
$$F\otimes_\cF {\det}_\cF(H_B(M)^+)\simeq \bigotimes_{v\in S_\infty(K)} {\det}_F(H^0(K_v,V^\ast(1)))$$
is an $\cO_F$-basis of the lattice $\bigotimes_{v\in S_\infty(K)} {\det}_{\cO_F}(H^0(K_v,T^\ast(1)))$. Then there is an $\cO_F$-basis
$$\fz_\gamma \in {\det}_{\cO_F}^{-1}(\rgamma(G_{K,S},T))$$
such that
$$\vartheta_0(\fz_\gamma)=\frac{L_S(M,0)}{\Omega_{\gamma,\delta}}\cdot \delta^\ast.$$
\end{conjecture}

\begin{remark}
One checks that the validity of Conjecture \ref{TNC general0} is independent of the choices of $S,T,\gamma,\delta$. 
\end{remark}

We shall next formulate the Tamagawa number conjecture in arbitrary analytic rank. 

Under Hypothesis \ref{TNChyp}, the Poitou-Tate exact sequence (\ref{TNCPT}) and the dual exponential map $\exp^\ast$ induce an isomorphism
$${\det}_F^{-1}(\rgamma(G_{K,S},V)) \simeq {\det}_F(H^1_f(K,V^\ast(1))) \otimes_F {\det}_F(H^1_f(K,V)) \otimes_\cF {\det}_\cF(t(M)^\ast).  $$
Combining this with the isomorphisms in Conjecture \ref{conj motivic}, we obtain an isomorphism 
$$\vartheta: {\det}_F^{-1}(\rgamma(G_{K,S},V)) \simeq F\otimes_\cF\left({\det}_\cF(H^1_f(K,M)) \otimes_\cF {\det}_\cF(H^1_f(K,M^\ast(1))) \otimes_\cF {\det}_\cF(t(M)^\ast) \right).  $$
Take $\cF$-bases
$$x \in {\det}_\cF(H^1_f(K,M)) \text{ and }y \in {\det}_\cF(H^1_f(K,M^\ast(1))).$$

\begin{conjecture}[The Tamagawa number conjecture for $(M,\cO_F)$]\label{TNC general1}
Assume Hypothesis \ref{TNChyp} and Conjecture \ref{conj motivic} (for $M$ and $M^\ast(1)$). Assume also that the generalized Deligne conjecture (Conjecture \ref{deligne general}) is true. Let $\gamma \in {\det}_\cF(H_B(M)^+)$ be as in Conjecture \ref{TNC general0}. Then there is an $\cO_F$-basis
$$\fz_\gamma \in {\det}_{\cO_F}^{-1}(\rgamma(G_{K,S},T))$$
such that
$$\vartheta(\fz_\gamma)= \frac{L_S^\ast(M,0)}{\Omega_{\gamma,\delta}R_{x,y}}\cdot x\otimes y\otimes \delta^\ast.$$
\end{conjecture}

\begin{remark}
One checks that the validity of Conjecture \ref{TNC general1} is independent of the choices of $S,T,\gamma,\delta, x,y$. 
\end{remark}

\subsection{Examples}

We give some basic examples of motives. 

\subsubsection{Tate motives}\label{sec ex tate}

For a number field $K$ and an integer $j\in \ZZ$, there is a Tate motive $M=h^0(\Spec K)(j)$. This is a motive defined over $K$ of rank one and weight $-2j$ with coefficients in $\QQ$. The realizations are the following. 
\begin{itemize}
\item For an embedding $\sigma: K \hookrightarrow \CC$, the $\sigma$-Betti realization is
$$H_\sigma(M):=\QQ(j):=(2\pi i)^j\QQ.$$
\item The $p$-adic realization is 
$$V_p(M):=\QQ_p(j).$$
\item The de Rham realization is
$$H_{\rm dR}(M):=K$$
with filtration
$${\rm Fil}^i H_{\rm dR}(M) := \begin{cases}
K &\text{if $i\leq -j$,}\\
0&\text{if $i>-j$.}
\end{cases}$$
\end{itemize}
The comparison isomorphisms are naturally defined. Note that $M^\ast(1)$ is identified with $h^0(\Spec K)(1-j)$. 
Hypothesis \ref{conj L} is satisfied. The $L$-function of $M$ is $L(M,s)=\zeta_K(s+j)$, where $\zeta_K(s)$ denotes the Dedekind zeta function of $K$. Conjecture \ref{conj motivic} is true with
$$H^0_f(K,M):=\begin{cases}
\QQ &\text{if $j=0$,}\\
0&\text{if $j\neq 0$,}
\end{cases} \text{ and }H^1_f(K,M):=\begin{cases}
\QQ\otimes_\ZZ \cO_K^\times &\text{if $j=1$,}\\
\QQ\otimes_\ZZ K_{2j-1}(K) &\text{if $j>1$,}\\
0&\text{if $j\leq 0$.}
\end{cases}$$
(The case $j>1$ is due to the Voevodsky-Rost theorem.) Conjecture \ref{conj order} is true by Borel's theorem. 

$M$ is critical if and only if $K$ is totally real and either $j$ is negative odd or positive even. Conjecture \ref{conj height} is trivially true since $H^1_f(K,M)=H^1_f(K,M^\ast(1))=0$ in this case.  
Conjecture \ref{deligne0} (which is the same as Conjecture \ref{deligne general} in this case) is true by the Klingen-Siegel theorem. Conjecture \ref{TNC general0} (which is the same as Conjecture \ref{TNC general1}) for negative odd $j$ is equivalent to the Lichtenbaum conjecture, which is proved by Wiles \cite[Thm. 1.6]{wiles}. Conjecture \ref{TNC general0} for positive even $j$ can be proved when $p$ is unramified in $K$ by using the functional equation (see \cite[Thm. 3.8(i)]{sbA2}). 

When $K$ is abelian over $\QQ$, the Tamagawa number conjecture (and its equivariant refinement) is proved by Burns-Greither \cite{BG}, Huber-Kings \cite{HK}, and Burns-Flach \cite{BFetnc2}. 

\subsubsection{Artin motives}\label{sec ex artin}

Let $\chi: G_K\to \CC^\times$ be a finite order character. Suppose that $\chi$ takes values in a number field $\cF$. Then there is an Artin motive $M=M(\chi)$ defined over $K$ of rank one and weight zero with coefficients in $\cF$. The realizations are the following. 
\begin{itemize}
\item For an embedding $\sigma: K \hookrightarrow \CC$, the $\sigma$-Betti realization is
$$H_\sigma(M):=\cF.$$
\item For $\lambda \in S_p(\cF)$, the $\lambda$-adic realization is 
$$V_\lambda(M):=\cF_\lambda(\chi),$$
i.e., $V_\lambda(M)$ is a one dimensional $\cF_\lambda$-vector space on which $G_K$ acts via $\chi$.
\item The de Rham realization is
$$H_{\rm dR}(M):=\cF\otimes_\QQ K$$
with filtration
$${\rm Fil}^i H_{\rm dR}(M) := \begin{cases}
\cF\otimes_\QQ K &\text{if $i\leq 0$,}\\
0&\text{if $i>0$.}
\end{cases}$$
\end{itemize}
If $\chi$ is the trivial character and $\cF=\QQ$, then we have $M=h^0(\Spec K)$. 
The $L$-function of $M$ is the Artin $L$-function $L(M,s)=L(\chi^{-1},s)$. When $K=\QQ$, the Tamagawa number conjecture for $M(j)$ (for any $j\in \ZZ$) is proved by Huber-Kings \cite{HK}.

\subsubsection{Elliptic curves}\label{ex sec ell}

Let $E$ be an elliptic curve defined over a number field $K$. Then we can consider the critical motive $M=h^1(E/K)(1)$ of rank two and weight $-1$ with coefficients in $\QQ$. The realizations are the following. 
\begin{itemize}
\item For an embedding $\sigma: K \hookrightarrow \CC$, the $\sigma$-Betti realization is
$$H_\sigma(M):=H_1(E^\sigma(\CC),\QQ)\simeq H^1(E^\sigma(\CC),\QQ(1)).$$
Here we set $E^\sigma:= E\times_{K,\sigma}\CC$. 
\item The $p$-adic realization is 
$$V_p(M):=H^1_{\text{\'et}}(E\times_K \overline \QQ, \QQ_p(1)).$$
(This is canonically isomorphic to $V_p(E):=\QQ_p\otimes_{\ZZ_p} T_p(E)$.)
\item The de Rham realization is
$$H_{\rm dR}(M):=H^1_{\rm dR}(E/K)$$
with filtration
$${\rm Fil}^i H_{\rm dR}(M) := {\rm Fil}^{i+1} H^1_{\rm dR}(E/K)=\begin{cases}
H^1_{\rm dR}(E/K) &\text{if $i<0$,}\\
\Gamma(E,\Omega_{E/K}^1) &\text{if $i=0$,}\\
0 &\text{if $i>0$.}
\end{cases}$$
\end{itemize}
The comparison isomorphisms are well-known. 
Hypothesis \ref{conj L}(i) and (ii) are satisfied. The $L$-function of $M$ is $L(M,s)=L(E/K,s+1)$, where $L(E/K,s)$ is the Hasse-Weil $L$-function for $E/K$.  Hypothesis \ref{conj L}(iii) is not known in general: when $K=\QQ$, it is a consequence of the Shimura-Taniyama conjecture proved by Wiles et al. Conjecture \ref{conj motivic} is true with
$$H^0_f(K,M):=0\text{ and }H^1_f(K,M):=\QQ\otimes_\ZZ E(K)$$
if the $p$-part of the Tate-Shafarevich group $\sha(E/K)$ is finite. Conjecture \ref{conj height} is satisfied with the N\'eron-Tate height pairing. It is known that the Tamagawa number conjecture for $(h^1(E/K)(1),\ZZ_p)$ (Conjecture \ref{TNC general1}) is equivalent to the $p$-part of the Birch and Swinnerton-Dyer formula. (See Proposition \ref{TNC equiv} in the case of analytic rank one.) 

We remark that $M$ is self-dual, i.e., $M^\ast(1)=M$. The meaning of this is that each realization of $M^\ast(1)$ is canonically isomorphic to that of $M$. 

\subsubsection{Algebraic varieties}

Let $X$ be a smooth projective variety defined over $K$. For $n, j \in \ZZ$ with $n\geq 0$, there is a motive $M=h^n(X)(j)$ defined over $K$ of weight $n-2j$ with coefficients in $\QQ$. The realizations are the following. 
\begin{itemize}
\item For an embedding $\sigma: K \hookrightarrow \CC$, the $\sigma$-Betti realization is
$$H_\sigma(M):=H^n(X^\sigma(\CC),\QQ(j)).$$
\item The $p$-adic realization is 
$$V_p(M):=H^n_{\text{\'et}}(X\times_K \overline \QQ, \QQ_p(j)).$$
\item The de Rham realization is
$$H_{\rm dR}(M):=H^n_{\rm dR}(X/K)$$
with filtration
$${\rm Fil}^i H_{\rm dR}(M) := {\rm Fil}^{i+j}H^n_{\rm dR}(X/K).$$
\end{itemize}
These realizations satisfy the axioms of Weil cohomology and have the well-known comparison isomorphisms. 
Properties of the $L$-function $L(M,s) $ are highly conjectural. For possible definitions of the motivic cohomology, see \cite[\S 6.5]{fontaineL}, \cite[\S 3.1]{BFetnc} for example.


The dual of $M$ is described as follows. 
Let $d:= \dim X$ and suppose $0\leq n\leq 2d$. Then we have $h^{2d}(X)(d) = h^0(\Spec K)$ and the Poincar\'e duality pairing
$$h^n(X) \times h^{2d-n}(X) \to h^{2d}(X)=h^0(\Spec K)(-d)$$
induces an identification $h^n(X)^\ast= h^{2d-n}(X)(d)$. Hence we have $M^\ast(1)= h^{2d-n}(X)(d+1-j)$. Also, by the hard Lefschetz theorem, we have $M^\ast(1)=h^n(X)(n+1-j)$.

\subsubsection{Modular forms}\label{sec ex modular}

Let $f=\sum_{n=1}^\infty a_n q^n \in S_{2r}(\Gamma_0(N))$ be a normalized newform of weight $2r$ and level $N$. We set $\cF:=\QQ(\{a_n\}_n)$, which is a totally real number field. Then there is a motive $M_f$ attached to $f$, which is defined over $\QQ$ of rank two and weight $2r-1$ with coefficients in $\cF$. The motive $M_f$ was first constructed by Scholl \cite{scholl}. We shall describe its realizations, following \cite[\S 2]{LV}.

Let $\cE_N \to X(N)$ be the universal generalized elliptic curve. The Kuga-Sato variety $X:=\widetilde \cE_N^{2r-2}$ is defined to be the canonical desingularization of $\cE_N^{2r-2}$ described in \cite{delignemodular}. Scholl constructed a certain projector $\Pi$ which acts on the cohomology of $X$ (see \cite[\S 2.3]{LV}). The motive $M_f$ is defined to be $\Pi \cdot h^{2r-1}(X)$, i.e., the realizations are the following. 
\begin{itemize}
\item The Betti realization is 
$$H_B(M_f):=\Pi\cdot H^{2r-1}(X(\CC), \cF).$$
\item For $\lambda \in S_p(\cF)$, the $\lambda$-adic realization is 
$$V_\lambda(M_f):= \Pi \cdot H_{\text{\'et}}^{2r-1}(X\times_\QQ \overline \QQ, \cF_\lambda).$$
\item The de Rham realization is 
$$H_{\rm dR}(M_f):=\Pi\cdot H_{\rm dR}^{2r-1}(X/\QQ)\otimes_\QQ \cF$$
with the usual filtration. 
\end{itemize}
The comparison isomorphisms are induced by those for $h^{2r-1}(X)$. Hypothesis \ref{conj L} is satisfied. The $L$-function of $M_f$ is $L(M_f,s)=L(f,s):=\sum_{n=1}^\infty a_n n^{-s}$. The motive $M_f(j)$ is critical when $1\leq j\leq 2r-1$. In this article, we are interested in the ``central critical twist" $M:=M_f(r)$, which is self-dual. Conjectures \ref{conj motivic} and \ref{conj height} for $M$ are not known in general: see \cite[\S\S 2.6 and 2.7]{LV}. Conjecture \ref{deligne0} for $M$ is known modulo comparison of periods (due to Shimura \cite{shimura}, \cite{shimuraperiod}). Kato essentially proves in \cite[Thm. 14.5(3)]{katoasterisque} that Conjecture \ref{TNC general0} for $M$ is implied by the Iwasawa main conjecture for $f$ (see \cite[Conj. 12.10]{katoasterisque}). Conjecture \ref{TNC general1} for $M$ in analytic rank one is studied by Longo-Vigni \cite[Thm. B]{LV}. 

\begin{remark}\label{rem det modular}
For $M:=M_f(r)$, note that $\bigwedge^2 M$ is identified with $h^0(\Spec \QQ)(1)$ (with coefficients in $\cF$). The meaning of this is that there are canonical isomorphisms
\begin{itemize}
\item ${\bigwedge}_\cF^2 H_B(M)\simeq \cF(1)(:=(2\pi i )\cF)$,
\item ${\bigwedge}_{\cF_\lambda}^2 V_\lambda(M) \simeq \cF_\lambda(1)$,
\item ${\bigwedge}_\cF^2 H_{\rm dR}(M)\simeq \cF$. 
\end{itemize}
%
This follows by noting that the Poincar\'e duality pairing
$$h^{2r-1}(X)(r) \times h^{2r-1}(X)(r) \to h^0(\Spec \QQ)(1)$$
is skew-symmetric (since $2r-1$ is odd). 
\end{remark}

\subsubsection{Hecke motives}\label{sec ex hecke}

We first review the definition and basic properties of Hecke characters. 

Let $K$ be an imaginary quadratic field. Let $\bA_K^\times$ be the id\`ele group of $K$. Let  $\widehat K^\times:= (\widehat \ZZ\otimes_\ZZ K)^\times$ be the finite id\`ele group. Note that $\bA_K^\times= \CC^\times \times \widehat K^\times$. 

Let $k,\ell \in \ZZ$. A Hecke character of $K$ of infinity type $(k,\ell)$ is a continuous homomorphism $\chi: \bA_K^\times/K^\times \to \CC^\times$ such that its restriction on $\CC^\times$ is given by $z\mapsto z^{-k}\overline z^{-\ell}$. (Note that our sign convention is opposite to \cite[\S 15.7]{katoasterisque}, \cite[\S 3.3]{CH}, but agrees with \cite{deshalit}, \cite{tsuji}.) We regard $\chi $ as a map $\bA_K^\times \to \CC^\times $ which is trivial on $K^\times$. 

$\chi$ is called anticyclotomic if $\chi$ is trivial on $\bA_\QQ^\times$. Note that, if $\chi$ is anticyclotomic, then its infinity type is of the form $(k,-k)$. 

We say that $\chi$ takes values in a number field $\cF$ if $\chi(\widehat K^\times)\subset \cF$. We set $\widehat \cO_K^\times:=\prod_{v<\infty}\cO_{K_v}^\times$. The conductor of $\chi$ is defined to be the largest ideal $\mathfrak{f}$ of $\cO_K$ such that the restriction of $\chi$ on $\widehat \cO_K^\times$ factors through $(\cO_K/\mathfrak{f})^\times$. (Note that ``$\mathfrak{f}$ is larger than $\mathfrak{g}$" means $\mathfrak{f}\mid \mathfrak{g}$.)

Let $I_K$ be the group of fractional ideals of $K$. Let $i_K: \widehat K^\times \to I_K; \ z \mapsto \prod_\fp \fp^{\ord_\fp(z_\fp)}$ be the natural surjection. Let $I_{K,\mathfrak{f}}\subset I_K$ be the group of fractional ideals of $K$ prime to $\mathfrak{f}$. For a Hecke character $\chi$ of conductor $\mathfrak{f}$ which takes values in $\cF$, we define
$$\widetilde \chi: I_{K,\mathfrak{f}} \to \cF^\times$$
in the following way: for $\mathfrak{a} \in I_{K,\mathfrak{f}}$, choose $z \in \widehat K^\times$ such that $i_K(z)=\mathfrak{a}$ and $z_\fp \equiv 1$ (mod $\fp^{\ord_\fp(\mathfrak{f})}$) for any $\fp \mid \mathfrak{f}$, and define $\widetilde \chi(\mathfrak{a}):=\chi(z)$. (One checks that this is well-defined.) Note that, for a principal ideal $(a) \in I_{K,\mathfrak{f}}$ such that $a \equiv 1$ (mod $\mathfrak{f}$), we have
$$\widetilde \chi ((a))= \chi|_{\widehat K^\times}(a) = \chi|_{\CC^\times}(a)^{-1}=a^k\overline a^{\ell}.$$
Conversely, for a given character $\widetilde \chi: I_{K,\mathfrak{f}}\to \cF^\times$ such that $\widetilde \chi((a)) = a^k\overline a^\ell$ for $a\in K^\times$ with $a\equiv 1$ (mod $\mathfrak{f}$), one can naturally construct a Hecke character $\chi$ of infinity type $(k,\ell)$ (see \cite[Chap. 0, \S 5]{scha}). Via this correspondence, we often identify $\widetilde \chi$ with $\chi$. In particular, we write $\chi(\mathfrak{a})$ instead of $\widetilde \chi(\mathfrak{a})$ for $\mathfrak{a}\in I_{K,\mathfrak{f}}$. As usual, we set $\chi(\mathfrak{a}):=0$ if $\mathfrak{a}$ is not prime to the conductor $\mathfrak{f}$. 

The Hecke $L$-function for $\chi$ is defined by
$$L(\chi,s):=\sum_{\mathfrak{a}}\frac{ \chi(\mathfrak{a})}{{\N}\mathfrak{a}^s},$$
where $\mathfrak{a}$ runs over all non-zero integral ideals of $K$ and we set ${\N}\mathfrak{a}:=\#(\cO_K/\mathfrak{a})$. We have an expression by the Euler product
$$L(\chi,s)= \prod_\fp (1-\chi(\fp){\N}\fp^{-s})^{-1},$$
where $\fp$ runs over all primes of $K$. 

Let ${\rm rec}_K: \bA_K^\times/K^\times \to G_K^{\rm ab}:=\Gal(K^{\rm ab}/K)$ denote the (arithmetically normalized) global reciprocity map. ($K^{\rm ab}$ denotes the maximal abelian extension of $K$.) A Hecke character of infinity type $(0,0)$ is identified with a finite order character $G_K\to \CC^\times$ via ${\rm rec}_K$. For a fixed embedding $\iota_p: \overline \QQ \hookrightarrow \overline \QQ_p$, we define the $p$-adic avatar $\chi_p:G_K \to \overline \QQ_p^\times$ of $\chi$ by
$$\chi_p({\rm rec}_K(z)) = \iota_p(\chi(z))z_\fp^{-k}z_{\barfp}^{-\ell},$$
where $z \in \widehat K^\times$ and $\fp$ denotes the prime of $K$ corresponding to $\iota_p$. 
By an abuse of notation, we often denote $\chi_p$ by $\chi$.

The most basic example of Hecke characters is the norm Hecke character
$$\N: \bA_K^\times/K^\times \to \CC^\times; \ z \mapsto \prod_v |z_v|_v^{-1},$$
where $v$ runs over all places of $K$ and $|\cdot|_v: K_v^\times \to \RR_{>0}$ denotes the normalized absolute value. The infinity type of $\N$ is $(1,1)$. The conductor of $\N$ is $(1)$, and the corresponding character of $I_K$ is given by
$$\widetilde \N: I_K \to \QQ^\times; \ \mathfrak{a} \mapsto {\N}\mathfrak{a}. $$
The $p$-adic avatar of $\N$ is the cyclotomic character $\chi_{\rm cyc}: G_K \to \ZZ_p^\times$. 

Assuming that $K$ has class number one, we fix an elliptic curve $A$ defined over $K$ with complex multiplication by $\cO_K$. Let $\psi=\psi_A$ be the associated Hecke character of infinity type $(1,0)$, which takes values in $K$ (see \cite[Prop. 7.41]{shimurabook}). A prime $\fp$ of $K$ divides the conductor of $\psi$ if and only if $A$ has bad reduction at $\fp$ (see \cite[Thm. II.1.8]{deshalit}). We have $\psi \overline \psi = \N$. If $\fp$ is the prime of $K$ corresponding to $\iota_p$, then $G_K$ acts on the $\fp$-adic Tate module $T_\fp(A):=\varprojlim_n A[\fp^n]$ via $\psi$. 
%

\vspace{3mm}

We now construct a motive. 
For a Hecke character $\chi$ of infinity type $(k,\ell)$ which takes values in $\cF$, there is a motive $M(\chi)$ attached to $\chi$ defined over $K$ of rank one and weight $-k-\ell$ with coefficients in $\cF$. We shall describe $M(\chi)$. 

By the observations above, it is natural to define $M(\N):= h^0(\Spec K)(1)$ and $M(\psi):=h^1(A)(1)$. For the general case, we write $\chi=\chi_0 \psi^k \overline \psi^\ell = \chi_0 \psi^{k-\ell} \N^\ell$ with a finite order character $\chi_0$. 
Then we define 
$$M(\chi):= M(\chi_0)\otimes (h^1(A)(1))^{\otimes (k-\ell)}(\ell),$$
where $M(\chi_0)$ is the Artin motive (see \S \ref{sec ex artin}). 
By $\chi^{-1}\N=\overline \chi \N^{1-k-\ell}$, we see that $M(\chi)^\ast(1)=M(\overline \chi)(1-k-\ell)$. 

In this article, we mainly consider Hecke characters of the form $\chi= \psi^k\overline \psi^\ell$ with $k>\ell$. In this case, $M(\chi)$ has coefficients in $K$, and the realizations are given as follows. 

\begin{itemize}
\item For an embedding $\sigma: K \hookrightarrow \CC$, the $\sigma$-Betti realization is
$$H_\sigma(M(\chi)):=H_1(A^\sigma(\CC),\QQ)^{\otimes(k-\ell)}(\ell).$$
($H_1(A^\sigma(\CC),\QQ)$ is a one-dimensional $K$-vector space and the tensor product is taken over $K$.) 
\item For $\fp \in S_p(K)$, the $\fp$-adic realization is 
$$V_\fp(M(\chi)):=V_\fp(A)^{\otimes (k-\ell)}(\ell).$$
(The tensor product is taken over $K_\fp$.) 
Here we set $V_\fp(A):= K_\fp \otimes_{\cO_{K_\fp}} T_\fp(A)$. 
\item The de Rham realization is
$$H_{\rm dR}(M(\chi)):=H^1_{\rm dR}(A/K)^{\otimes (k-\ell)}.$$
($H^1_{\rm dR}(A/K)$ is a free $K\otimes_\QQ K$-module of rank one and the tensor product is taken over $K\otimes_\QQ K$.) The filtration is given by
$${\rm Fil}^i H_{\rm dR}(M(\chi)) := \begin{cases}
H^1_{\rm dR}(A/K)^{\otimes (k-\ell)} &\text{if $i\leq -k$,}\\
\Gamma(A,\Omega_{A/K}^1)^{\otimes (k-\ell)} &\text{if $-k<i \leq -\ell$,}\\
0 &\text{if $i>-\ell$.}
\end{cases}$$
\end{itemize}

Hypothesis \ref{conj L} is satisfied for the motive $M(\chi)$. 
Note that the $L$-function of $M(\chi)$ is 
$$L(M(\chi),s)=L(\chi^{-1},s) = L(\overline \psi^{k-\ell},s+k) = L(\psi^{\ell-k},s+\ell).$$
We know that $M(\chi)$ is critical (i.e., ``$\chi^{-1}$ is critical" in the sense of \cite[\S II.1.1]{deshalit}) if and only if $\ell \leq 0 <k$. (Note that we suppose $k>\ell$.) 
The Deligne conjecture (Conjecture \ref{deligne0}) is proved by Goldstein-Schappacher \cite{GS} (see also \cite[Thm. II.4.3]{tsuji}). A large part of the Tamagawa number conjecture in the critical case (Conjecture \ref{TNC general0}) is proved by Kato \cite[Chap. III]{katolecture}, \cite[\S 15]{katoasterisque}, Guo \cite{guo}, Han \cite{han} and Tsuji \cite[Thm. II.10.4]{tsuji} as an application of explicit reciprocity laws and the Iwasawa main conjecture proved by Rubin \cite{rubinimag}. For results in the non-critical case, see \cite{kings}, \cite{bars}. 

When $(k,\ell)=(1,0)$ (i.e., $M(\chi)=h^1(A)(1)$), Conjecture \ref{conj order} in analytic rank zero is the well-known result due to Coates-Wiles \cite{CW} and Rubin \cite{rubintateshafarevich}. A large part of Conjecture \ref{TNC general0} is due to Rubin \cite{rubinimag}, and it has recently been solved by Burungale-Flach \cite{burungaleflach} completely. (See \cite[Prop. 2.3]{burungaleflach} for an explicit interpretation of Conjecture \ref{TNC general0}.) In analytic rank one, Conjecture \ref{TNC general1} is essentially proved by Rubin \cite{rubinimag} when $A$ is defined over $\QQ$, and in a  more general case it has recently been proved by Castella \cite{castellaTNC}. 

\vspace{3mm}

We shall give an explicit description of the period of $M:=M(\chi)$, which is used in this article. Assume that $M$ is critical and let 
$$\alpha_{M}: \RR\otimes_\QQ H_B(M)^+\xrightarrow{\sim} \RR\otimes_\QQ t(M)\simeq \RR\otimes_\QQ (\Gamma(A,\Omega_{A/K}^1)^{\otimes(k-\ell)} )^\ast$$
be the period isomorphism. Note that $H_B(M)^+$ is identified with $H_1(A(\CC),\QQ)^{\otimes (k-\ell)}(\ell)$. 
Let $\omega_A \in \Gamma(A,\Omega_{A/K}^1)$ be a N\'eron differential and $\gamma_A \in H_1(A(\CC),\ZZ)$ an $\cO_K$-basis. We define the complex CM period by
$$\Omega_\infty:=\int_{\gamma_A}\omega_A.$$
We set
$$\gamma:= (2\pi i )^\ell \gamma_A^{\otimes(k-\ell)} \in H_1(A(\CC),\QQ)^{\otimes(k-\ell)}(\ell)= H_B(M)^+$$
and 
$$\delta^\ast := \omega_A^{\otimes (k-\ell)} \in \Gamma(A,\Omega_{A/K}^1)^{\otimes (k-\ell)}.$$
Let $\delta \in t(M)$ be the dual basis of $\delta^\ast$. Then the period of $M$
$$\Omega_{\gamma,\delta} \in (\RR\otimes_\QQ K)^\times =\CC^\times$$
with respect to $\gamma$ and $\delta$ is defined by
 $$\alpha_{M}(\gamma)=\Omega_{\gamma,\delta}\cdot \delta.$$

\begin{proposition}\label{hecke period}
We have
$$\Omega_{\gamma,\delta}= \pm \Omega_\infty^{k-\ell} \left(\frac{\sqrt{D_K}}{2\pi}\right)^{-\ell}.$$
Here $-D_K < 0$ denotes the discriminant of $K$. 
\end{proposition}

\begin{proof}
This is proved in \cite[Prop. II.4.10]{tsuji}. Note that $A(\CC/\cO_K)$ in \cite[Prop. II.2.6]{tsuji} is $\sqrt{D_K}/2$. 
\end{proof}

\section{The Birch and Swinnerton-Dyer formula in analytic rank one}\label{sec app}

Let $E$ be an elliptic curve defined over $\QQ$ with conductor $N$. Let $K$ be an imaginary quadratic field with odd discriminant $-D_K <-3$. (We do not assume $K$ has class number one.) We assume the Heegner hypothesis: every prime divisor of $N$ splits in $K$. Let $p$ be an odd prime number which does not divide $ND_K$. (Namely, $p$ is unramified in $K$ and $E$ has good reduction at $p$.)

In this appendix, we give a proof of the following result. 

\begin{theorem}\label{thman1}
Assume that $p$ splits in $K$. If $\ord_{s=1}L(E/K,s) =1$, then the Tamagawa number conjecture for the pair $(h^1(E/K)(1),\ZZ_p)$ is implied by the Iwasawa main conjecture for the Bertolini-Darmon-Prasanna $p$-adic $L$-function (Conjecture \ref{IMC}).
\end{theorem}

The proof is given in \S \ref{sec pf1}. 

\begin{remark}
By the well-known Gross-Zagier-Kolyvagin theorem, we know that $\ord_{s=1}L(E/K,s) =1$ implies the finiteness of $\sha(E/K)$ and ${\rm rank}(E(K))=1$. 
\end{remark}

\begin{remark}\label{rem JSW}
Since the Tamagawa number conjecture for the pair $(h^1(E/K)(1),\ZZ_p)$ is equivalent to the $p$-part of the Birch and Swinnerton-Dyer formula for $E/K$ (see Proposition \ref{TNC equiv} below), Theorem \ref{thman1} is essentially proved by Jetchev-Skinner-Wan in \cite[\S 7.4.1]{JSW}. Our argument does not rely on the ``anticyclotomic control theorem" in \cite[Thm. 3.3.1]{JSW}.
\end{remark}

We also have the following result, which we prove in \S \ref{sec ks}. 

\begin{theorem}\label{thmheeg}
Assume that $E$ has good ordinary reduction at $p$. If $\ord_{s=1}L(E/K,s) =1$, then the Tamagawa number conjecture for the pair $(h^1(E/K)(1),\ZZ_p)$ is implied by the Heegner point main conjecture (see Conjecture \ref{HIMC} below). 
\end{theorem}

Note that in this result we do not need to assume that $p$ splits in $K$, but the ``ordinary" assumption is imposed.

We set some notations used in this appendix. 
Let $S$ be the finite set of places of $K$ consisting of the infinite place and the primes dividing $pN$. We set $T:=T_p(E)$ and $V:=\QQ_p\otimes_{\ZZ_p} T$. Let $K_\infty/K$ be the anticyclotomic $\ZZ_p$-extension. We set $\Gamma:=\Gal(K_\infty/K)$, $\Lambda := \ZZ_p[[\Gamma]] $ and $\TT:= \Lambda \otimes_{\ZZ_p}T$. 

Our idea of the proofs of Theorems \ref{thman1} and \ref{thmheeg} is to construct a $\Lambda$-basis
$$\fz_S \in {\det}_\Lambda^{-1}(\rgamma(G_{K,S},\TT))$$
which interpolates the value $L'(E/K,1)$. In the case of Theorem \ref{thman1} (i.e., when $p$ splits in $K$), this is the basis constructed in Proposition \ref{coeff}. In the case of Theorem \ref{thmheeg}, we construct $\fz_S$ by using Heegner points (see \S \ref{sec ks} below).

\subsection{The Tamagawa number conjecture for elliptic curves}

We review the formulation of the Tamagawa number conjecture for the pair $(h^1(E/K)(1),\ZZ_p)$ in the case $\ord_{s=1}L(E/K,s)=1$. This is a special case of Conjecture \ref{TNC general1} for the motive given in \S \ref{ex sec ell}. 

\begin{lemma}\label{kslemma}
Assume $\ord_{s=1}L(E/K,s)=1$. (In particular, $\sha(E/K)$ is finite and ${\rm rank}(E(K))=1$.) Then we have $H^2(G_{K,S},V)=0$ and there is a canonical exact sequence
$$0\to \QQ_p\otimes_\ZZ E(K) \to H^1(G_{K,S},V)\to \QQ_p\otimes_\QQ \Gamma(E,\Omega_{E/K}^1) \to \QQ_p\otimes_\ZZ E(K)^\ast \to 0.$$
\end{lemma}

\begin{proof}
This is proved in \cite[Lem. 5.1]{ks}. Note that the exact sequence is obtained by combining the Poitou-Tate exact sequence
\begin{equation}\label{PT}
0\to H^1_f(K,V)\to H^1(G_{K,S},V) \to \bigoplus_{v\mid p}H^1_{/f}(K_v,V)
\to H^1_f(K,V)^\ast \to 0,
\end{equation}
the Kummer isomorphism
$$\QQ_p\otimes_\ZZ E(K) \simeq H^1_f(K,V),$$
and the dual exponential map
$$\exp^\ast : \bigoplus_{v\mid p}H^1_{/f}(K_v,V)\xrightarrow{\sim} \QQ_p\otimes_\QQ \Gamma(E,\Omega_{E/K}^1). $$
\end{proof}

In the following, we assume $\ord_{s=1}L(E/K,s)=1$. Then by Lemma \ref{kslemma} we have a canonical isomorphism
\begin{equation}\label{ks isom}
\vartheta: {\det}_{\QQ_p}^{-1}(\rgamma(G_{K,S},V)) \xrightarrow{\sim} \QQ_p\otimes_\QQ \left(E(K)\otimes_\ZZ E(K) \otimes_\ZZ {\bigwedge}_\QQ^2 \Gamma(E,\Omega_{E/K}^1)\right).
\end{equation}

Take a $\ZZ$-basis $\gamma \in {\bigwedge}_\ZZ^2 H_1(E(\CC),\ZZ)$.
Take also a non-zero element $\delta \in {\bigwedge}_\QQ^2 \Gamma(E,\Omega_{E/K}^1)^\ast$. Let 
$$\alpha: \RR\otimes_\QQ {\bigwedge}_\QQ^2 H_1(E(\CC),\QQ)\xrightarrow{\sim} \RR\otimes_\QQ {\bigwedge}_\QQ^2 \Gamma(E,\Omega_{E/K}^1)^\ast$$
be the period map. We define a period $\Omega_{E,\gamma,\delta}\in \RR^\times$ with respect to $\gamma$ and $\delta$ by
$$\alpha(\gamma)=\Omega_{E,\gamma,\delta}\cdot \delta. $$

\begin{remark}\label{rem neron}
One can take $\delta \in {\bigwedge}_\QQ^2 \Gamma(E,\Omega_{E/K}^1)^\ast$ such that
$$\Omega_{E,\gamma,\delta} = \sqrt{D_K}^{-1}\Omega_{E/K},$$
where $\Omega_{E/K}$ denotes the N\'eron period for $E/K$. 
\end{remark}

Take any non-torsion element $x\in E(K)$. Let 
$$\langle -,-\rangle_\infty: E(K) \times E(K) \to \RR$$
be the N\'eron-Tate height pairing. By the Gross-Zagier formula \cite{GZ}, one can show that 
$$\frac{L'(E/K,1)}{\Omega_{E,\gamma,\delta}\langle x,x \rangle_\infty}  \in \QQ.$$
(This means that the generalized Deligne conjecture (Conjecture \ref{deligne general}) is true in this case.)
The Tamagawa number conjecture is stated as follows. 

\begin{conjecture}[The Tamagawa number conjecture for $(h^1(E/K)(1),\ZZ_p)$]\label{TNC2}
Assume $\ord_{s=1}L(E/K,s)=1$. Then there is a $\ZZ_p$-basis
$$\fz_E \in {\det}_{\ZZ_p}^{-1}(\rgamma(G_{K,S},T))$$
such that
\begin{equation}\label{theta image}
\vartheta(\fz_E) = \frac{L_S'(E/K,1)}{\Omega_{E,\gamma,\delta}\langle x,x \rangle_\infty} \cdot x\otimes x \otimes \delta^\ast.
\end{equation}
Here $\vartheta$ is the isomorphism in (\ref{ks isom}) and $L_S(E/K,s)$ denotes the $L$-function for $E/K$ with the Euler factors at primes in $S$ removed. 
\end{conjecture}

\begin{remark}\label{TNC independence}
One easily sees that the validity of Conjecture \ref{TNC2} is independent of the choices of $\gamma,\delta$ and $x$. 
\end{remark}

The following is well-known. 

\begin{proposition}\label{TNC equiv}
Conjecture \ref{TNC2} is equivalent to the $p$-part of the Birch and Swinnerton-Dyer formula, i.e., 
$$\frac{L'(E/K,1)}{\sqrt{D_K}^{-1}\Omega_{E/K}R_{E/K}}\cdot \ZZ_p = \frac{\# \sha(E/K)\cdot {\rm Tam}(E/K)}{\# E(K)_{\rm tors}^2}\cdot \ZZ_p \quad (\text{in }\QQ_p). $$
Here $\Omega_{E/K}, R_{E/K}$, and ${\rm Tam}(E/K)$ denote the N\'eron period, the N\'eron-Tate regulator, and the product of Tamagawa factors for $E/K$ respectively. 
\end{proposition}

For the reader's convenience, we give a proof of Proposition \ref{TNC equiv} in \S \ref{sec pf equiv}. 

\begin{remark}\label{rem GZ}
Fix a modular parametrization $\phi: X_0(N)\to E$ and let $y_K \in E(K)$ be the Heegner point defined by $\phi$. Let $c_\phi$ be the Manin constant and set 
$$z_K:= c_\phi^{-1}y_K \in \QQ\otimes_\ZZ E(K).$$
Then the Gross-Zagier formula \cite{GZ} states that
$$L'(E/K,1)=\sqrt{D_K}^{-1}\Omega_{E/K} \cdot \langle z_K,z_K\rangle_\infty.$$
If we take $\delta$ so that $\Omega_{E,\gamma,\delta} = \sqrt{D_K}^{-1}\Omega_{E/K}$, then (\ref{theta image}) is equivalent to
$$\vartheta(\fz_E)= {\rm Eul}_S\cdot z_K\otimes z_K \otimes\delta^\ast. $$
Here ${\rm Eul}_S$ denotes the product of Euler factors at finite places $v\in S$ such that ${\rm Eul}_S\cdot  L'(E/K,1)=L_S'(E/K,1)$.
\end{remark}

\subsection{Proof of Proposition \ref{TNC equiv}} \label{sec pf equiv}

By Remarks \ref{TNC independence} and \ref{rem neron}, we may assume that $x$ is a generator of $E(K)_{\rm tf}:=E(K)/E(K)_{\rm tors}$ so that $R_{E/K}=\langle x,x\rangle_\infty$ and that $\delta$ is taken so that $\sqrt{D_K}^{-1}\Omega_{E/K}=\Omega_{E,\gamma,\delta}$. It is sufficient to show the equality
\begin{equation}\label{BSD equality}
\vartheta\left( {\det}_{\ZZ_p}^{-1}(\rgamma(G_{K,S},T))\right) = \ZZ_p\cdot {\rm Eul}_S\cdot\frac{\# \sha(E/K)\cdot {\rm Tam}(E/K)}{\# E(K)_{\rm tors}^2}\cdot x\otimes x\otimes \delta.
\end{equation}
Here ${\rm Eul}_S$ is as in Remark \ref{rem GZ}. (The formula (\ref{BSD equality}) is actually proved in \cite[(5.4.6)]{sanoderived} without assuming ${\rm rank}(E(K))=1$, but we provide a proof for the sake of completeness.) 

We set $W:=V/T=E[p^\infty]$. The Pontryagin dual of a $\ZZ_p$-module $M$ is denoted by $M^\vee$. We identify $W$ with $T^\vee(1)$ via the Weil pairing. By the Poitou-Tate duality, we have an exact sequence
\begin{multline}\label{poitoutate}
0\to H^1_f(K,T)\to H^1(G_{K,S},T) \to \bigoplus_{v\in S}H^1_{/f}(K_v,T)\\
\to H^1_f(K,W)^\vee \to H^2(G_{K,S},T)\to \bigoplus_{v\in S}H^2(K_v,T)\to H^0(K,W)^\vee\to 0. 
\end{multline}
We use the following facts. 
\begin{itemize}
\item $H^1_f(K,T)\simeq \ZZ_p\otimes_\ZZ E(K)$. 
\item $H^1_{/f}(K_v,T)=0$ if $v\nmid p$. 
\item $H^1_{/f}(K_v,T)\simeq (\ZZ_p\widehat \otimes E(K_v) )^\ast$ if $v\mid p$.
\item There is a canonical exact sequence
$$0\to \sha(E/K)[p^\infty]^\vee \to H^1_f(K,W)^\vee \to (\ZZ_p\otimes_\ZZ E(K))^\ast \to 0.$$
\item $H^2(K_v,T)\simeq H^0(K_v,W)^\vee \simeq E(K_v)[p^\infty]^\vee$ for $v\in S$.
\item $H^0(K,W)= E(K)[p^\infty]$. 
\end{itemize}

We set $E(K_p):=\bigoplus_{v\mid p} E(K_v)$. By (\ref{poitoutate}) and the facts above, we obtain a canonical isomorphism
\begin{align*}
{\det}_{\ZZ_p}^{-1}(\rgamma(G_{K,S},T)) &= {\det}_{\ZZ_p}(H^1(G_{K,S},T)) \otimes_{\ZZ_p} {\det}_{\ZZ_p}^{-1}(H^2(G_{K,S},T)) \\
&\simeq  {\det}_{\ZZ_p}(\ZZ_p\otimes_\ZZ E(K)) \otimes_{\ZZ_p} {\det}_{\ZZ_p}((\ZZ_p\widehat \otimes E(K_p))^\ast)\otimes_{\ZZ_p} {\det}_{\ZZ_p}^{-1}(\sha(E/K)[p^\infty]^\vee )\\
& \quad \otimes_{\ZZ_p} {\det}_{\ZZ_p}^{-1}((\ZZ_p\otimes_\ZZ E(K))^\ast) \otimes_{\ZZ_p} {\det}_{\ZZ_p}^{-1}\left( \bigoplus_{v\in S}E(K_v)[p^\infty]^\vee\right) \otimes_{\ZZ_p}{\det}_{\ZZ_p}(E(K)[p^\infty]) \\
&\simeq \ZZ_p\cdot {\rm Eul}_p\cdot \frac{\#\sha(E/K)}{\# E(K)_{\rm tors}^2} \cdot \prod_{v\mid N}\# E(K_v)[p^\infty] \cdot x\otimes x\otimes \delta,
\end{align*}
where the last isomorphism is due to Lemma \ref{lem formal} below. 
By definition, this isomorphism is induced by $\vartheta$. The desired equality (\ref{BSD equality}) follows by noting that the equality
$$\prod_{v\mid N} \# E(K_v)[p^\infty] = {\rm Eul}_N\cdot {\rm Tam}(E/K)$$
holds up to a $p$-adic unit, where ${\rm Eul}_N$ denotes the product of Euler factors at $v\mid N$. Hence we have completed the proof of Proposition \ref{TNC equiv}. \qed

\begin{lemma}\label{lem formal}
The dual exponential map
$$\exp^\ast: (\QQ_p\widehat \otimes E(K_p))^\ast \to \QQ_p\otimes_\QQ \Gamma(E,\Omega_{E/K}^1)$$
induces an isomorphism
$${\det}_{\ZZ_p}^{-1}(\ZZ_p\widehat \otimes E(K_p))= {\det}_{\ZZ_p}((\ZZ_p\widehat \otimes E(K_p))^\ast)\otimes_{\ZZ_p} {\det}_{\ZZ_p}^{-1}(E(K_p)[p^\infty]^\vee) \xrightarrow{\sim} \ZZ_p\cdot {\rm Eul}_p\cdot \delta^\ast \subset \QQ_p\otimes_\QQ{\bigwedge}_\QQ^2\Gamma(E,\Omega_{E/K}^1),$$
where ${\rm Eul}_p$ denotes the product of Euler factors at $v\mid p$. 
\end{lemma}

\begin{proof}
We first remark that the equality 
$${\det}_{\ZZ_p}^{-1}(\ZZ_p\widehat \otimes E(K_p))= {\det}_{\ZZ_p}((\ZZ_p\widehat \otimes E(K_p))^\ast)\otimes_{\ZZ_p} {\det}_{\ZZ_p}^{-1}(E(K_p)[p^\infty]^\vee) $$
follows from the following general fact: for a finitely generated $\ZZ_p$-module $M$, we have 
$${\det}_{\ZZ_p}^{-1}(M)={\det}_{\ZZ_p}(\rhom_{\ZZ_p}(M,\ZZ_p))$$
and 
$$H^i(\rhom_{\ZZ_p}(M,\ZZ_p)) =\begin{cases}
M^\ast:=\Hom_{\ZZ_p}(M,\ZZ_p) &\text{ if $i=0$,}\\
\Ext_{\ZZ_p}^1(M,\ZZ_p)\simeq M[p^\infty]^\vee &\text{ if $i=1$,}\\
0&\text{ if $i\neq 0,1$.}
\end{cases}$$

We set $K_p:=\QQ_p\otimes_\QQ K=K_\fp\oplus K_{\barfp}$ and $\cO_{K_p}:=\ZZ_p\otimes_\ZZ \cO_K$. Let $\QQ_p\otimes_\QQ\Gamma(E,\Omega_{E/K}^1) = \Gamma(E,\Omega_{E/K_\fp}^1)\oplus \Gamma(E,\Omega_{E/K_{\barfp}}^1)$ be the canonical decomposition induced by $\QQ_p\otimes_\QQ K = K_\fp \oplus K_{\barfp}$. For $\fq\in \{\fp,\barfp\}$, let $\omega_\fq \in \Gamma(E,\Omega_{E/K_\fq}^1)$ be the image of the fixed N\'eron differential $\omega \in \Gamma(E,\Omega_{E/\QQ}^1)$. {\color{black}One checks that the $\ZZ_p$-submodule of $\QQ_p\otimes_\QQ  {\bigwedge}_{\QQ}^2\Gamma(E,\Omega_{E/K}^1)$ generated by $\delta^\ast$ coincides with that generated by $\omega_\fp \wedge \omega_{\barfp}$}. 
Hence it is sufficient to show that the formal logarithm associated with $\omega$
$$\log_\omega: \QQ_p\widehat \otimes E(K_p) \xrightarrow{\sim} K_p$$
induces an isomorphism
\begin{equation}\label{formal eulp}
{\det}_{\ZZ_p}(\ZZ_p\widehat \otimes E(K_p)) \xrightarrow{\sim}  {\rm Eul}_p^{-1}\cdot {\bigwedge}_{\ZZ_p}^2 \cO_{K_p}.
\end{equation}

For $v\mid p$, we have an exact sequence
$$0\to E_1(K_v)\to E(K_v)\to E(\FF_v)\to 0,$$
where $\FF_v$ denotes the residue field of $v$. From this, we obtain
$${\det}_{\ZZ_p}(\ZZ_p\widehat \otimes E(K_p)) \simeq \left(\prod_{v\mid p}\# E(\FF_v)\right)^{-1} \cdot {\bigwedge}_{\ZZ_p}^2 E_1(K_p).$$
Since $\log_\omega$ induces an isomorphism $E_1(K_p)\xrightarrow{\sim} p\cO_{K_p}$, we have
$${\det}_{\ZZ_p}(\ZZ_p\widehat \otimes E(K_p)) \xrightarrow{\sim}  p^2\left(\prod_{v\mid p}\# E(\FF_v)\right)^{-1}  \cdot {\bigwedge}_{\ZZ_p}^2 \cO_{K_p}.$$
We obtain (\ref{formal eulp}) if we note ${\rm Eul}_p= p^{-2}\prod_{v\mid p}\# E(\FF_v)$. 
\end{proof}

\subsection{Proof of Theorem \ref{thman1}}\label{sec pf1}

In this subsection, we give a proof of Theorem \ref{thman1}. We assume that $p$ splits in $K$. Let $(p)=\fp \barfp$ be the decomposition of $p$ in $K$. 

\subsubsection{Construction of a basis}

We assume the Iwasawa main conjecture for the BDP $p$-adic $L$-function (Conjecture \ref{IMC}). (We take $f$ to be the modular form of weight two corresponding to $E$.) Let 
$$\fz_S \in {\det}_\Lambda^{-1}(\rgamma(G_{K,S},\TT)) $$
be the $\Lambda$-basis constructed in Proposition \ref{coeff} by using the basis $\fz_\fp$ in Conjecture \ref{IMC} and the local epsilon element $\varepsilon_\fp$ in (\ref{def local epsilon}).  We define a $\ZZ_p$-basis
$$\fz_E \in {\det}_{\ZZ_p}^{-1}(\rgamma(G_{K,S},T))$$
to be the image of $\fz_S$ under the natural surjection
$${\det}_\Lambda^{-1}(\rgamma(G_{K,S},\TT)) \xrightarrow{a\mapsto a\otimes 1} {\det}_\Lambda^{-1}(\rgamma(G_{K,S},\TT)) \otimes_{\Lambda}\ZZ_p \simeq {\det}_{\ZZ_p}^{-1}(\rgamma(G_{K,S},T)).$$
Here the last isomorphism follows from
\begin{equation*}\label{control}
\rgamma(G_{K,S},\TT)\lotimes_\Lambda \ZZ_p\simeq \rgamma(G_{K,S},T)
\end{equation*}
(see \cite[Prop. 1.6.5(3)]{FK}). 

We take $\delta$ as in Remark \ref{rem neron}. By Remark \ref{rem GZ}, Theorem \ref{thman1} is reduced to the following. 

\begin{theorem}\label{thman1reduced}
Assume $\ord_{s=1}L(E/K,s)=1$ and the Iwasawa main conjecture (Conjecture \ref{IMC}) for $E$. Then we have
$$\vartheta(\fz_E)= {\rm Eul}_S\cdot z_K\otimes z_K \otimes\delta^\ast$$
up to a unit in $\ZZ_p^\times$. 
\end{theorem}


\subsubsection{Proof of Theorem \ref{thman1reduced}}

Let $\log=\log_{\omega,\fp}: \QQ_p\otimes_\ZZ E(K) \simeq \QQ_p\widehat \otimes E(K_\fp) \xrightarrow{\sim} \QQ_p$ be the formal logarithm associated with the fixed N\'eron differential $\omega$. We set $a_p:=p+1-\# E(\FF_p)$. 

Let $L_\fp^{\rm BDP} \in \Lambda^{\rm ur}$ be the BDP $p$-adic $L$-function for $E$ as in \S \ref{sec IMC}. 
The key is to use the ``$p$-adic Gross-Zagier formula" established in \cite[Thm. 5.13]{BDP}:
\begin{equation}\label{pGZ}
{\bf 1}(L_\fp^{\rm BDP}) = \left(\frac{1-a_p + p}{p}\right)^2 \log(z_K)^2.
\end{equation}
Here ${\bf 1}: \Lambda^{\rm ur}\twoheadrightarrow \widehat \ZZ_p^{\rm ur}$ denotes the augmentation map. 

We first need the following lemma. 

\begin{lemma}\label{lem skinner}
The localization map at $\fp$
$$H^1(G_{K,S},V)\to H^1(K_\fp,V)$$
is an isomorphism. 
\end{lemma}

\begin{proof}
For $\fq \in \{\fp, \barfp\}$, we set
$$\widetilde \rgamma_\fq(K,V):= {\rm Cone}\left(\rgamma(G_{K,S},V)\to \rgamma(K_\fq,V)\right)[-1].$$
Then by \cite[Lem. 4.6]{sanoderived} we have $H^1(\widetilde \rgamma_\fq(K,V)) =0$. (Note that ${\rm rank}(E(K))=1$ by the assumption $\ord_{s=1}L(E/K,s)=1$.) Since we have $H^2(\widetilde \rgamma_\fp(K,V))\simeq H^1(\widetilde \rgamma_{\barfp}(K,V))^\ast$ by duality, we also have $H^2(\widetilde \rgamma_\fp(K,V))=0$. This proves the lemma. 
\end{proof}

\begin{proof}[Proof of Theorem \ref{thman1reduced}]

Let 
$${\rm loc}^2_\fp: {\det}_{\QQ_p}^{-1}(\rgamma(G_{K,S},V)) = {\bigwedge}_{\QQ_p}^2 H^1(G_{K,S},V) \to{\bigwedge}_{\QQ_p}^2 H^1(K_\fp,V) $$
be the map induced by the localization map at $\fp$. Note that this is an isomorphism by Lemma \ref{lem skinner}. 

Let 
$${\bf 1}(\varepsilon_\fp) \in \widehat \QQ_p^{\rm ur}\otimes_{\QQ_p} {\det}_{\QQ_p}(\rgamma(K_\fp,V)) =\widehat \QQ_p^{\rm ur} \otimes_{\QQ_p}\Hom_{\QQ_p} \left( {\bigwedge}_{\QQ_p}^2 H^1(K_\fp,V) , \QQ_p\right)$$
be the image of the local epsilon element $\varepsilon_\fp \in \Lambda^{\rm ur}\otimes_\Lambda {\det}_\Lambda(\rgamma(K_\fp,\TT))$ in (\ref{def local epsilon}) under the natural surjection
$$\Lambda^{\rm ur}\otimes_\Lambda {\det}_\Lambda(\rgamma(K_\fp,\TT))\xrightarrow{a\mapsto a\otimes 1} \Lambda^{\rm ur}\otimes_\Lambda {\det}_\Lambda(\rgamma(K_\fp,\TT))\otimes_{\Lambda}\ZZ_p \simeq \widehat \ZZ_p^{\rm ur}\otimes_{\ZZ_p}{\det}_{\ZZ_p}(\rgamma(K_\fp,T)).$$
Then we have 
$${\bf 1}(\varepsilon_\fp)  \in \Hom_{\QQ_p}\left({\bigwedge}_{\QQ_p}^2 H^1(K_\fp,V),\QQ_p\right)$$
by the property in \cite[Conj. 3.4.3(iv)]{FK} (since $\tau_p$ acts trivially on ${\bigwedge}_{\ZZ_p}^2 T\simeq \ZZ_p(1)$).

By the construction of $\fz_E$, we have
\begin{equation}\label{by constr}
{\bf 1}(\varepsilon_\fp)\circ {\rm loc}^2_\fp (\fz_E) =  {\rm Eul}_N\cdot {\bf 1}(L_\fp^{\rm BDP}),
\end{equation}
where ${\rm Eul}_N$ is the product of Euler factors at $v\mid N$, which satisfies
$${\rm Eul}_S= {\rm Eul}_N \left(\frac{1-a_p + p}{p}\right)^2.$$

Consider the isomorphism
$$\lambda_p: \QQ_p\otimes_\QQ \left(E(K)\otimes_\ZZ E(K) \otimes_\ZZ {\bigwedge}_\QQ^2 \Gamma(E,\Omega_{E/K}^1)\right)\xrightarrow{\sim} \QQ_p; \ 1\otimes (x\otimes y \otimes \delta^\ast )\mapsto \log(x)\log(y).$$
By (\ref{by constr}) and Lemma \ref{keylem2} below, we have
\begin{equation}\label{lambda theta}
\lambda_p(\vartheta(\fz_E)) = {\rm Eul}_N\cdot {\bf 1}(L_\fp^{\rm BDP})
\end{equation}
up to a unit in $\ZZ_p^\times$. 

By (\ref{pGZ}) and (\ref{lambda theta}), we obtain
$$\lambda_p(\vartheta(\fz_E)) = {\rm Eul}_S\cdot \log(z_K)^2 = \lambda_p({\rm Eul}_S\cdot z_K\otimes z_K \otimes \delta^\ast)$$
up to a unit in $\ZZ_p^\times$. Since $\lambda_p$ is an isomorphism, this proves Theorem \ref{thman1reduced}. 
\end{proof}

It is now sufficient to prove the following. 

\begin{lemma}\label{keylem2}
We have
$${\bf 1}(\varepsilon_\fp) \circ {\rm loc}^2_\fp = \lambda_p\circ \vartheta \text{ in }\Hom_{ \QQ_p} \left({\det}_{\QQ_p}^{-1}(\rgamma(G_{K,S},V)),\QQ_p\right)$$
up to a unit in $\ZZ_p^\times$. 
\end{lemma}

\begin{proof}
One can check that the following diagram is commutative:
$$\xymatrix{
{\bigwedge}_{\QQ_p}^2 H^1(G_{K,S},V)\ar[d]_{{\rm loc}_\fp^2} \ar[r]_-f^-\sim& H^1_f(K,V) \otimes_{\QQ_p} H^1_f(K,V) \otimes_{\QQ_p} H^1_{/f}(K_\fp,V)\otimes_{\QQ_p} H^1_{/f}(K_{\barfp},V) \ar[d]^{h:={\rm loc}_\fp \otimes {\rm loc}_{\barfp} \otimes \id \otimes \id}\\
{\bigwedge}_{\QQ_p}^2 H^1(K_\fp,V) \ar[r]_-g^-\sim&H^1_f(K_\fp,V) \otimes_{\QQ_p} H^1_f(K_{\barfp},V) \otimes_{\QQ_p} H^1_{/f}(K_\fp,V)\otimes_{\QQ_p} H^1_{/f}(K_{\barfp},V).
}$$
Here $f$ is induced by the Poitou-Tate exact sequence (\ref{PT}), and $g$ by the exact sequence
\begin{equation}\label{tautological}
0 \to H^1_f(K_\fp,V)\to H^1(K_\fp,V) \to H^1_{/f}(K_\fp,V)\to 0
\end{equation}
and the local Tate duality isomorphism
\begin{equation}\label{f local duality}
H^1_f(K_{\barfp},V)\otimes_{\QQ_p} H^1_{/f}(K_{\barfp},V)\xrightarrow{\sim}\QQ_p.
\end{equation}
(We identify $V$ with $V^\ast(1)$ via the Weil pairing.) 

Let 
$$\log= \log_\omega: H^1_f(\QQ_p,V)\xrightarrow{\sim}\QQ_p$$
and
$$\exp^\ast=\exp^\ast_\omega: H^1_{/f}(\QQ_p,V)\xrightarrow{\sim} \QQ_p$$
be the logarithm and the dual exponential maps associated with $\omega$ respectively. 
{\color{black}By the definition of $\vartheta$ in (\ref{ks isom}), one sees that the map $\lambda_p\circ \vartheta$ coincides up to a unit in $\ZZ_p^\times$ with the following composition map:
\begin{align*}
{\bigwedge}_{\QQ_p}^2 H^1(G_{K,S},V) &\xrightarrow{h\circ f}H^1_f(K_\fp,V) \otimes_{\QQ_p} H^1_f(K_{\barfp},V) \otimes_{\QQ_p} H^1_{/f}(K_\fp,V)\otimes_{\QQ_p} H^1_{/f}(K_{\barfp},V)\\
&\xrightarrow{\log\otimes \log \otimes \exp^\ast \otimes \exp^\ast} \QQ_p.
\end{align*}
}

On the other hand, by the definition of $\varepsilon_{\QQ_p,\xi}(V)$ in \cite[\S 3.3.1]{FK}, one checks that the map ${\bf 1}(\varepsilon_\fp)$ coincides with 
\begin{equation*}
{\bigwedge}_{\QQ_p}^2 H^1(K_\fp,V) \simeq H^1_f(K_\fp,V)\otimes_{\QQ_p}H^1_{/f}(K_\fp,V) \xrightarrow{\log \otimes \exp^\ast} \QQ_p,
\end{equation*}
where the first isomorphism is induced by the exact sequence (\ref{tautological}). 

It is now sufficient to show the commutativity of the diagram
$$\xymatrix{
{\bigwedge}_{\QQ_p}^2 H^1(K_\fp,V)\ar[d]_-\simeq \ar[r]^-g&H^1_f(K_\fp,V) \otimes_{\QQ_p} H^1_f(K_{\barfp},V) \otimes_{\QQ_p} H^1_{/f}(K_\fp,V)\otimes_{\QQ_p} H^1_{/f}(K_{\barfp},V) \ar[dl]^-{\id \otimes \log \otimes \id \otimes \exp^\ast}\\
H^1_f(K_\fp,V)\otimes_{\QQ_p}H^1_{/f}(K_\fp,V) .&
}$$
However, this follows from the fact that the isomorphism (\ref{f local duality}) is given by
$$a\otimes b \mapsto \log(a)\exp^\ast(b).$$
Hence we have completed the proof. 
\end{proof}

\subsection{Proof of Theorem \ref{thmheeg}}\label{sec ks}

In this subsection, we give a proof of Theorem \ref{thmheeg}. 

We set some notations. 
In this subsection, we assume that $E$ has good ordinary reduction at $p$ so that there is a canonical exact sequence of $G_{\QQ_p}$-representations:
$$0\to F^+T \to T \to F^-T\to 0.$$
We set $F^{\pm} V:= \QQ_p\otimes_{\ZZ_p} F^{\pm}T$. We also set $F^{\pm}\TT:=\Lambda\otimes_{\ZZ_p} F^{\pm} T$. 

Let $\alpha \in \ZZ_p^\times$ be the unit root of $x^2-a_px +p$, where $a_p:=p+1-\#E(\FF_p)$. Let $\beta:=p/\alpha$ be the other root. 

We set $K_p:=\QQ_p\otimes_\QQ K = \bigoplus_{v\mid p}K_v$. We write $\rgamma(K_p,-)$ for $\bigoplus_{v\mid p} \rgamma(K_v,-)$. Similarly, we write $H^i(K_p,-)$ for $\bigoplus_{v\mid p}H^i(K_v,-)$. 

We define a Selmer complex by
\begin{equation}\label{def sel greenberg}
\widetilde \rgamma_f(K,\TT):= {\rm Cone}\left(\rgamma(G_{K,S},\TT)\to \rgamma(K_p,F^-\TT)\oplus \bigoplus_{v\mid N} \rgamma_{/{\rm ur}}(K_v,\TT) \right)[-1].
\end{equation}

\subsubsection{The Heegner point main conjecture}

Let 
$$z_\infty=(z_n)_n \in \varprojlim_n \ZZ_p\otimes_\ZZ E(K_n)$$
be the system of regularized Heegner points as in \cite[\S 2.5]{BD}, where $K_n$ denotes the $n$-th layer of the anticyclotomic $\ZZ_p$-extension $K_\infty/K$. We normalize it by multiplying $c_\phi^{-1}$. Then we have
\begin{equation}\label{z0formula}
z_0 = \begin{cases}
(1-\alpha^{-2})z_K &\text{if $p$ is inert,}\\
(1-\alpha^{-1})^2 z_K &\text{if $p$ is split,}
\end{cases}
\end{equation}
where $z_K$ is the Heegner point as in Remark \ref{rem GZ}. 
We regard $z_\infty \in \widetilde H^1_f(K,\TT):=H^1(\widetilde \rgamma_f(K,\TT))$ via the Kummer map. 

If the $\Lambda$-rank of $\widetilde H^1_f(K,\TT)$ is one, then we have a canonical isomorphism as in \cite[(5.2.1)]{ks}:
\begin{equation}\label{ks can isom}
Q(\Lambda) \otimes_\Lambda {\det}_\Lambda^{-1}(\rgamma_f(K,\TT)) \simeq Q(\Lambda )\otimes_\Lambda \widetilde H^1_f(K,\TT) \otimes \widetilde H^1_f(K,\TT)
\end{equation}

The following formulation of the Heegner point main conjecture is given in \cite[Prop. 5.12]{ks}. 

\begin{conjecture}[The Heegner point main conjecture]\label{HIMC}
The $\Lambda$-rank of $\widetilde H^1_f(K,\TT)$ is one and there is a $\Lambda$-basis
$$\widetilde \fz_\infty \in {\det}_\Lambda^{-1}(\rgamma_f(K,\TT))$$
such that its image under (\ref{ks can isom}) is $z_\infty\otimes z_\infty$. 
\end{conjecture}

\subsubsection{Local epsilon elements}

Let
$$\exp^\ast: {\bigwedge}_{\QQ_p}^2 H^1(K_p, V) \to \QQ_p\otimes_\QQ {\bigwedge}_\QQ^2 \Gamma(E,\Omega_{E/K}^1)\xrightarrow{\delta^\ast \mapsto 1} \QQ_p$$
be the map induced by the dual exponential map, where $\delta$ is as in Remark \ref{rem neron}. This induces an isomorphism
$$\exp^\ast: {\det}_{\QQ_p}^{-1}(\rgamma(K_p,F^-V))= {\bigwedge}_{\QQ_p}^2 H^1(K_p, F^-V)\xrightarrow{\sim} \QQ_p.$$

The basis $\varepsilon_p$ in the following lemma plays the role of a local epsilon element (see Remark \ref{rem choice} below). 

\begin{lemma}\label{lem choice}
There is a $\Lambda$-basis
$$\varepsilon_p \in {\det}_\Lambda^{-1}(\rgamma(K_p,F^-\TT))$$
such that its image under the map
$${\det}_\Lambda^{-1}(\rgamma(K_p,F^-\TT))\xrightarrow{a\mapsto a\otimes 1} {\det}_\Lambda^{-1}(\rgamma(K_p,F^-\TT))\otimes_\Lambda \ZZ_p \simeq {\det}_{\ZZ_p}^{-1}(\rgamma(K_p,F^-T))\xrightarrow{\exp^\ast}\QQ_p$$
is $(1-\alpha^{-2})^{-1}(1-\beta^{-2})$ (resp. $(1-\alpha^{-1})^{-2}(1-\beta^{-1})^2$) if $p$ is inert (resp. split) in $K$. 
\end{lemma}

\begin{proof}
Note first that we have an equality up to a unit in $\ZZ_p^\times$:
$$p^{-2} \prod_{v\mid p} \# E(\FF_v)^{-1} = \begin{cases}
(1-\alpha^{-2})^{-1}(1-\beta^{-2}) &\text{if $p$ is inert,} \\
(1-\alpha^{-1})^{-2}(1-\beta^{-1})^2 &\text{if $p$ is split}.
\end{cases}$$
Here $\FF_v$ denotes the residue field of $v$. Hence it is sufficient to show that the image of the map
$$\exp^\ast: {\det}_{\ZZ_p}^{-1}(\rgamma(K_p,F^-T))\to \QQ_p$$
is generated over $\ZZ_p$ by $p^{-2} \prod_{v\mid p} \# E(\FF_v)^{-1}$. 

We shall abbreviate $H^i(K_p,-)$ to $H^i(-)$ and similarly for $\rgamma(K_p,-)$. By the exact sequence
$$0\to H^1(F^+T) \to H^1(T) \to H^1(F^-T) \to H^2(F^+T) \to H^2(T) \to 0,$$
we obtain a canonical isomorphism
$${\det}_{\ZZ_p}^{-1}(\rgamma(F^-T)) \simeq {\det}_{\ZZ_p}(H^1_f(T)/H^1(F^+T)) \otimes_{\ZZ_p} {\det}_{\ZZ_p}(H^2(T^+)) \otimes_{\ZZ_p} {\det}_{\ZZ_p}^{-1}(\rgamma_{/f}(T)).$$
By \cite[Prop. 2.5]{greenberg}, we have
$${\det}_{\ZZ_p}(H^1_f(T)/H^1(F^+T))  \simeq \ZZ_p\cdot \prod_{v\mid p}\# E(\FF_v)^{-1}.$$
Also, since $H^2(T^+)\simeq H^0(F^-V/F^-T)^\vee= \bigoplus_{v\mid p} E(\FF_v)[p^\infty]^\vee$ by duality, we have
$$ {\det}_{\ZZ_p}(H^2(T^+)) \simeq \ZZ_p\cdot \prod_{v\mid p}\# E(\FF_v)^{-1}.$$
By Lemma \ref{lem formal}, we see that the image of 
$$\exp^\ast: {\det}_{\ZZ_p}^{-1}(\rgamma_{/f}(T))\to \QQ_p$$
is $\ZZ_p\cdot p^{-2} \prod_{v\mid p} \#E(\FF_v)$. This proves the claim. 
\end{proof}

\begin{remark}\label{rem choice}
We can choose $\varepsilon_p$ in Lemma \ref{lem choice} in a natural way as follows. 

Suppose that $p$ is split and let $(p)=\fp \barfp$ be the decomposition. Then we have
$${\det}_\Lambda(\rgamma(K_p,F^-\TT)) = {\det}_\Lambda(\rgamma(K_\fp,F^-\TT))\otimes_\Lambda {\det}_\Lambda(\rgamma(K_{\barfp},F^-\TT)).$$
Let $\varphi_p$ and $\tau_p$ be as in Proposition \ref{phiepsilon}. For $\fq \in \{\fp, \barfp\}$, let
$$\varepsilon_\fq:=\varepsilon_{\Lambda,\xi}(K_\fq, F^-\TT) \in {\det}_\Lambda(\rgamma(K_\fq,F^-\TT))\otimes_\Lambda F^-\TT\otimes_\Lambda \Lambda_a$$
be the local epsilon element (see \cite[Conj. 3.4.3]{FK} and Remark \ref{rem isom basis}), where 
$$\Lambda_{a}:= \{ x\in \Lambda^{\rm ur} \mid \varphi_p(x)= a^{-1}x\},$$
and $a \in \Lambda^\times$ denotes the image of $\tau_p \in G_{\QQ_p}$ under the map $G_{\QQ_p} \to {\rm Aut}(F^-\TT)\simeq \Lambda^\times$. (Since $F^-\TT$ is a rank one representation, the validity of the local epsilon conjecture is proved by Kato \cite{katolecture2} and Venjakob \cite{venjakob}.) By fixing a $\Lambda$-basis of $F^-\TT\otimes_\Lambda \Lambda_a$, we can regard 
$$\varepsilon_\fq \in {\det}_\Lambda(\rgamma(K_\fq,F^-\TT)).$$
We may take $\varepsilon_p$ to be the dual of $\varepsilon_\fp \otimes \varepsilon_{\barfp}$. 

Suppose next that $p$ is inert. If we consider the induced module $F^-\TT_K:= \ZZ_p[\Gal(K/\QQ)] \otimes_{\ZZ_p} F^-\TT$, then we have
$${\det}_\Lambda(\rgamma(K_p,F^-\TT))={\det}_\Lambda(\rgamma(\QQ_p, F^-\TT_K)).$$
Note that $F^-\TT_K\simeq F^-\TT \oplus F^-\TT(\chi_K)$, where $\chi_K :G_{\QQ_p}\to \{\pm 1\}$ denotes the quadratic character corresponding to $K_p$. Hence we know the validity of the  local epsilon conjecture in this case and we similarly get a basis $\varepsilon_p$. 
\end{remark}

\subsubsection{Construction of a basis}

By the definition of the Selmer complex (\ref{def sel greenberg}) and Lemma \ref{lem euler}, we have a canonical isomorphism
\begin{equation}\label{sel triang}
{\det}_\Lambda^{-1}(\rgamma(G_{K,S},\TT)) \simeq {\det}_\Lambda^{-1}(\widetilde \rgamma_f(K,\TT)) \otimes_\Lambda {\det}_\Lambda^{-1}(\rgamma(K_p,F^-\TT)).
\end{equation}
Assuming the Heegner point main conjecture (Conjecture \ref{HIMC}), we have a $\Lambda$-basis $\widetilde \fz_\infty \in {\det}_\Lambda^{-1}(\widetilde \rgamma_f(K,\TT))$ and define
$$\fz_S \in {\det}_\Lambda^{-1}(\rgamma(G_{K,S},\TT))$$
to be the $\Lambda$-basis which corresponds to $\widetilde \fz_\infty\otimes \varepsilon_p$ under the isomorphism (\ref{sel triang}), where $\varepsilon_p$ is as in Lemma \ref{lem choice}. We then define a $\ZZ_p$-basis
$$\fz_E \in {\det}_{\ZZ_p}^{-1}(\rgamma(G_{K,S},T))$$
to be the image of $\fz_S$ under the natural surjection
$${\det}_\Lambda^{-1}(\rgamma(G_{K,S},\TT)) \xrightarrow{a\mapsto a\otimes 1} {\det}_\Lambda^{-1}(\rgamma(G_{K,S},\TT)) \otimes_{\Lambda}\ZZ_p \simeq {\det}_{\ZZ_p}^{-1}(\rgamma(G_{K,S},T)).$$

\begin{remark}\label{rem heeg element}
We define a ``$\Lambda$-adic Heegner element" by
$$z_\infty^{\rm Hg} := \Theta(\fz_S) \in {\bigcap}_\Lambda^2 H^1(G_{K,S},\TT),$$
where $\Theta$ is as in (\ref{Theta}). This construction is an improvement of that in \cite[\S 5.2.3]{ks}, since we make the choice of a basis of ${\det}_\Lambda^{-1}(\rgamma(K_p,F^-\TT))$ specific as in Lemma \ref{lem choice}. 
\end{remark}

\subsubsection{Completion of the proof}

We now prove Theorem \ref{thmheeg}. 

\begin{proof}[Proof of Theorem \ref{thmheeg}]

By Remarks \ref{rem neron} and \ref{rem GZ}, it is sufficient to show the equality
$$\vartheta(\fz_E)= {\rm Eul}_S\cdot z_K\otimes z_K \otimes\delta^\ast.$$

By the construction of $\fz_E$, we have
$$\vartheta(\fz_E) = {\rm Eul}_N\cdot \begin{cases}
(1-\alpha^{-2})^{-1}(1-\beta^{-2}) z_0\otimes z_0 \otimes \delta^\ast &\text{if $p$ is inert,}\\
(1-\alpha^{-1})^{-2}(1-\beta^{-1})^2 z_0\otimes z_0 \otimes \delta^\ast &\text{if $p$ is split.}
\end{cases}$$
Here ${\rm Eul}_N$ is the product of Euler factors at $v\mid N$. Noting that the product of Euler factors at $v\mid p$ is given by 
$${\rm Eul}_p= p^{-2}\prod_{v\mid p}\# E(\FF_v) = \begin{cases}
(1-\alpha^{-2})(1-\beta^{-2}) &\text{if $p$ is inert,}\\
(1-\alpha^{-1})^2 (1-\beta^{-1})^2 &\text{if $p$ is split},
\end{cases}$$
we see by (\ref{z0formula}) that 
$$\vartheta(\fz_E)= {\rm Eul}_S\cdot z_K\otimes z_K \otimes\delta^\ast.$$
This completes the proof. 
\end{proof}

\begin{acknowledgments}
The author would like to thank Francesc Castella and Shinichi Kobayashi for very helpful discussions. He would also like to thank Kentaro Nakamura for his interest and helpful comments. 
\end{acknowledgments}

\end{document}